\documentclass[aos,preprint]{imsart}
\usepackage{amsmath,amscd}
\usepackage{amssymb,soul}
\usepackage{amsthm,color}

\usepackage{graphicx}

\usepackage{mathrsfs}
\usepackage{hyperref}

\usepackage[colorinlistoftodos]{todonotes}

\graphicspath{{/EPSF/}{../figures/}{figures/}}

\newtheorem{Theorem}{Theorem}[section]

\newtheorem{Lemma}[Theorem]{Lemma}

\newtheorem{Proposition}[Theorem]{Proposition}

\newtheorem{condition}[Theorem]{Condition}
\newtheorem{Definition}[Theorem]{Definition}

\newtheorem{Remark}[Theorem]{Remark}

\numberwithin{equation}{section}

\def \Rm { {\mathbb R}}
\def \N {\mathcal{N}}

\def \<{\langle}
\def \>{\rangle}

\def \p{\partial}

\def \dim{{\mbox {dim}}\,}

\def\Vol{\mbox{Vol}}
\def\V{\mbox{Var}}

\def\R\re
\def\V{\bf V}

\def \re{{\mathbb R}}

\def \V{{\bf V}}

\newcommand{\norm}[1]{\lVert #1 \rVert}

\newcommand{\eps}{\varepsilon}

\newcounter{sidenote}

\begin{document}

\title{Efficient Nonparametric Bayesian \\ Inference  for X-ray transforms} 

\author{\fnms{Fran\c{c}ois Monard,} \ead[label=e1]{fmonard@ucsc.edu}}
\author{\fnms{Richard} \snm{Nickl}\ead[label=e2]{r.nickl@statslab.cam.ac.uk}}
\and
\author{\fnms{Gabriel P.} \snm{Paternain}\ead[label=e3]{g.p.paternain@dpmms.cam.ac.uk}}
\runauthor{F. Monard, R. Nickl, G. Paternain}
\runtitle{Bayesian inference for X-ray transforms}

\affiliation{University of California Santa Cruz \& University of Cambridge}
\address{University of California at Santa Cruz \\ Department of Mathematics \\
\printead{e1}}

\address{University of Cambridge \\ Department of
Pure Mathematics \\ ~and Mathematical Statistics\\
\printead{e2}\\
\printead{e3}}

%\author[F. Monard]{Fran\c{c}ois Monard}
%\address{Department of Mathematics, University of California, Santa Cruz, CA 95064}
%\email{fmonard@ucsc.edu}

%\author[R. Nickl]{Richard Nickl}
%\address{ Department of Pure Mathematics and Mathematical Statistics,University of Cambridge, Cambridge CB3 0WB, UK}
%\email {r.nickl@statslab.cam.ac.uk}

%\author[G.P. Paternain]{Gabriel P. Paternain}
%\address{ Department of Pure Mathematics and Mathematical Statistics, University of Cambridge, Cambridge CB3 0WB, UK}
%\email {g.p.paternain@dpmms.cam.ac.uk}

\date{\today}

\begin{abstract}
We consider the statistical inverse problem of recovering a function $f: M \to \mathbb R$, where $M$ is a smooth compact Riemannian manifold with boundary, from measurements of general $X$-ray transforms $I_a(f)$ of $f$, corrupted by additive Gaussian noise. For $M$ equal to the unit disk with `flat' geometry and $a=0$ this reduces to the standard Radon transform, but our general setting allows for anisotropic media $M$ and can further model local `attenuation' effects -- both highly relevant in practical imaging problems such as SPECT tomography. We study a nonparametric Bayesian inference method based on standard Gaussian process priors for $f$. The posterior reconstruction of $f$ corresponds to a Tikhonov regulariser with a reproducing kernel Hilbert space norm penalty that does not require the calculation of the singular value decomposition of the forward operator $I_a$. We prove Bernstein-von Mises theorems for a large family of one-dimensional linear functionals of $f$, and they entail that posterior-based inferences such as credible sets are valid and optimal from a frequentist point of view. In particular we derive the asymptotic distribution of smooth linear functionals of the Tikhonov regulariser, which attains the semi-parametric information lower bound. The proofs rely on an invertibility result for the `Fisher information' operator $I_a^*I_a$ between suitable function spaces, a result of independent interest that relies on techniques from microlocal analysis. We illustrate the performance of the proposed method via simulations in various settings. 

\end{abstract}



\maketitle

\section{Introduction}

The Radon transform and its variants play a key role in image reconstruction problems, with important applications in physics, engineering and other areas of scientific imaging. The classical case is where a function $f$ in $\mathbb{R}^2$ is reconstructed from integrals over straight lines:
$$
Rf(s,\omega) = \int_{-\infty}^{\infty} f(s\omega + t\omega^{\perp}) \,dt, \quad s \in \mathbb{R},~ \omega \in S^1,
$$
where $\omega^{\perp}$ is the rotation of $\omega$ by $90$ degrees counterclockwise. Often it is natural to confine the function $f$ to a bounded subset $M$ of Euclidean space such as the unit disk, where integrals are now taken along lines connecting boundary points of $M$. Such transforms constitute the basis for imaging methods such as {\it computerised tomography} (CT) and {\it positron emission tomography} (PET), and their mathematical properties are well studied \cite{Hel, N86}. 

Two generalisations of the standard Radon transform are important in applications: a) to model an attenuation or absorption effect within $M$, for example regions of different levels of biological activity in the physical medium $M$, and b) to model anisotropy or physical heterogeneity of $M$, for instance when `shortest travel times' of waves through the earth follow  geodesics of a non-Euclidean metric. The methods used for a) form the basis for SPECT imaging techniques (see for instance \cite{BGH79,Ku}) and b) occurs naturally in seismology, helioseismology and acoustic tomography problems, to mention a few \cite{DJHP,C92,MW79,Sh}. Both effects can be tackled by the general notion of \textit{attenuated geodesic $X$-ray transforms}  that are given by the formula
\begin{equation}\label{fkac}
I_af(x,v) = \int_0^{\tau(x,v)} f(\gamma_{(x,v)}(t)) e^{\int_0^t a(\gamma_{(x,v)}(s))ds}dt,
\end{equation}
where $(x,v)$ parametrises the set of geodesics $\gamma_{(x,v)}$ through $M$, and where $\tau(x,v)$ is the `exit time' of the geodesic started at a point $x$ at the boundary $\partial M$ in the direction of $v$ -- see Section \ref{xray} for precise definitions. The case $a=0$ corresponds to the case when no attenuation is present, and the `geometry' of $M$ is naturally encoded in the set of geodesics.

The mathematical inverse problem here is to recover $f$ from the line integral values $I_a(f)$ along all geodesics. Explicit reconstruction formulas are available in some specific settings: in the case of the flat disk in $\mathbb R^2$ and when $a=0$ this was proved in Radon's celebrated 1917 paper \cite{R17}, and it has been shown in the last 2 decades that explicit inversion formulas hold also in a variety of other more involved settings, namely, ``simple" geometries, see \cite{N02}, \cite{PU04} and the paper \cite{M14} on numerical implementation. It is, however, generally not clear how the inversion step should be done in case of observations corrupted by statistical noise. The general approach to noisy inverse problems that can be found in the statistical literature is typically based on obtaining a singular value decomposition (SVD) of the forward operator $I_a$ and to then construct a procedure based on spectral regularisation, see, e.g., the papers \cite{JS90, GP00, CGPT02, C08, KKLPP10, KVV11, R13}, just to mention a few. For the standard Radon transform such methods have been suggested in the seminal paper by Johnstone and Silverman \cite{JS90} where the SVD basis is given by Zernike polynomials. Another approach consists in ``rebinning'' fan-beam data into parallel data, for which regularisation methods are well-understood thanks to the Fourier-slice theorem; see \cite{N86}. However, neither approach adapts well to more general X-ray transforms: with attenuation and/or general geodesics, the SVD can rarely be computed analytically; nor is rebinning an option, as the space of geodesics is not homogeneous in general, and this removes the possibility of regularisation methods based on parallel geometry.

In the present paper we follow the Bayesian approach to inverse problems \cite{S10, DLSV13, DS16} and study a basic nonparametric inference method built around a standard Gaussian prior for the unknown function $f$ which does not require the identification of the SVD basis of $I_a$. We show how this method can be implemented in a standard way and the resulting maximum a posteriori (MAP) point estimates correspond to a Tikhonov regulariser with a common Sobolev norm penalty, where the Sobolev norm is defined in a classical way (and not implicitly via the SVD of $I_a$). We prove a  Bernstein-von Mises theorem that entails asymptotic normality of various `semi-parametric aspects' of the posterior distribution. From it we deduce in particular asymptotic normality and statistical efficiency of the plug-in Tikhonov regulariser for linear integral functionals $\langle f, \psi \rangle_{L^2},$ where $\psi$ is any smooth test function on $M$. In other words we establish that the semi-parametric information bound in this problem is attained by a standard regularisation method that does not require the calculation of the SVD basis. The proof is based on a combination of ideas from Bayesian nonparametric statistics \cite{CN13, CN14, C14} with an inversion result for the `Fisher information'  operator $I_a^*I_a$ between suitable function spaces (here $I^*_a$ is a natural adjoint operator defined below). 

Heuristically (by analogy to the finite-dimensional linear model) the semi-parametric information lower bound for  inference on $\langle f, \psi \rangle_{L^2}$ should be $\mathcal I_\psi=\|I_a(I_a^*I_a)^{-1}\psi\|_{L^2}^2$, but in our infinite-dimensional setting it has to be clarified for which $\psi$ this quantity is well-defined. In Section \ref{pxray} we invert the `Fisher information' operator by solving the homogeneous Dirichlet boundary value problem for the pseudo-differential operator $I^*_aI_a$, using techniques from micro-local analysis. The mapping properties we deduce imply in particular Theorem \ref{main0}c below, which rigorously establishes that $\mathcal I_\psi$ exists for all smooth $\psi$ (and equals the information lower bound). In our inversion result for $I_a^*I_a$, a key analytical difficulty, explained in more detail at the outset of Section \ref{pxray}, arises at the boundary $\partial M$ of $M$: for example, when applied to smooth (say constant) functions, $I_a^*I_a$ can generate singularities at $\partial M$. And even if one assumes that the unknown $f$, and thus relevant test functions $\psi$, are supported \textit{strictly within} $M$, an application of $(I_a^*I_a)^{-1}$ to such $\psi$ will produce a function that is \textit{fully} supported in $M$ (in view of the non-locality of the inverse operator). Dealing with boundary issues can therefore not be dispensed with. These non-locality effects can also be seen in numerical simulations (Example 3 below).

The connection to partial differential equation (PDE) models just mentioned deserves a final remark: For $M$ a bounded domain in $\mathbb R^d$ with smooth boundary $\partial M$, consider the transport equation
\begin{equation} \label{PDE}
v \cdot \nabla_x u(x, v) + a(x) u(x,v) = -f(x), ~~x \in M,\;\; v \in S^{d-1},
\end{equation}
subject to the boundary condition $u(x,v)=0$ for $x \in \partial M, v \cdot \nu(x) \geq 0$, where $\nu(x)$ is the outer normal at $x$. Along each straight line the transport equation \eqref{PDE} becomes an ordinary differential equation that 
is easily solved to find that the influx trace of $u$ is precisely the function $I_a(f)$.
Our results can thus be cast into the setting of Bayesian inference for parameters of partial differential equations (here $f$) from noisy observations of their solutions (here $I_a(f)$), studied by A. Stuart and others in the inverse problems literature, see \cite{S10, DS16} for an overview and \cite{NS17, N17} for recent related theoretical contributions for parabolic and elliptic PDEs.

This article is organised as follows: In Section \ref{xray} we introduce general $X$-ray transforms and state the invertibility theorem for the information operator. In Section \ref{method} we propose a Bayesian nonparametric method for inference from noisy $X$-ray transform data, and in Section \ref{bvmsec} we give the theoretical results about the performance of the Bayes method and the associated Tikhonov regulariser. All proofs can be found in subsequent sections.

\section{Main results}

\subsection{Geodesic $X$-ray transforms and an inversion result for the information operator}\label{xray}

In this section we introduce the geodesic X-ray transform $I$ of a compact Riemannian manifold with boundary as well as the attenuated version $I_{a}$. Our main objective is to establish mapping properties for the normal (information) operator $I_{a}^*I_{a}$.
% and to put the theory into the framework of the transmission condition as developed in \cite{Ho0,Grubb2}.  

The geodesic X-ray transform acts on functions defined on a compact oriented $d$-dimensional Riemannian manifold $(M,g)$ with boundary $\p M$ ($d\geq 2$). In essence, it integrates the function along all possible geodesics running between boundary points. To define the transform with precision we need to introduce some language that conveniently describes the geodesics on a manifold. Geodesics in a Riemannian manifold can be defined in many ways, but for our purposes it suffices to say that they are curves that locally minimize the distance between two points. It turns out that they obey a second order ordinary differential equation on $M$ and thus a geodesic is uniquely determined by its initial position and velocity (i.e. a point in phase space).
Geodesics travel at constant speed, so we might as well from now on fix the speed to be one.
It is hence convenient to pack positions and velocities together in what we call the {\it unit sphere bundle} $SM$.
This consists of pairs $(x,v)$, where $x\in M$ and $v$ is a tangent vector at $x$ with norm $|v|_g=1$, where $g$ is the inner product in the tangent space at $x$ (i.e. the Riemannian metric).

Unit tangent vectors at the boundary of $M$ constitute the boundary $\partial SM$ of $SM$ and will play a special role. Specifically
\[\partial SM:=\{(x,v)\in SM:\;x\in\partial M\}.\]
We will need to distinguish those tangent vectors pointing inside (``influx boundary'') and those pointing outside (``outflux boundary''), so we define
 two subsets of $\p SM$
$$\p_{\pm}SM:=\{(x,v)\in \p SM : \pm\<v,\nu(x)\>_g\leq 0\},$$
where $\nu(x)$ is the outward unit normal vector on $\p M$ at $x$.
% It is easy to see that
%$$\p_+SM\cap\p_-SM=S(\p M).$$

Given $(x,v)\in SM$, we denote by $\gamma_{x,v}:\mathbb{R}\to M$ the unique geodesic with $\gamma_{x,v}(0)=x$ and $ \frac{d\gamma_{x,v}}{dt}(0)=v$ and let $\tau(x,v)$ be the first time when the geodesic $\gamma_{x,v}$  exits $M$.

We say that $(M,g)$ is {\it non-trapping} if $\tau(x,v)<\infty$ for all $(x,v)\in SM$. In this paper we will work exclusively with non-trapping manifolds and this is all we need to define the geodesic X-ray transform. Let $C^\infty(W)$ denote the set of infinitely differentiable functions on a manifold $W$.

\begin{Definition}{\rm 
The \emph{geodesic X-ray transform} of a function $f \in C^{\infty}(M)$ is the function $If:\partial_{+}SM\to\mathbb{R}$ given by
\begin{equation*}
If(x,v)=\int_{0}^{\tau(x,v)}f(\gamma_{x,v}(t)) \,dt,\quad
(x,v)\in \partial_{+} SM.
\end{equation*}
}
\end{Definition}

In order to obtain good mapping properties for $I$, we need additional conditions on $M$. The second condition that we will impose is that $M$ has {\it strictly convex boundary}, i.e. the second fundamental form $\Pi_{x}(v,v):=\langle \nabla_{v}\nu,v\rangle_{g}$, for $v$ any tangent vector at $x$, is positive definite for all $x\in \partial M$. 
This ensures that $I:C^{\infty}(M)\rightarrow C^{\infty}(\partial_{+} SM)$ since strict convexity of the boundary implies $\tau\in C^{\infty}(\partial_{+}SM)$ \cite[Lemma 4.1.1]{Sharafudtinov1994}.

Effectively, the influx boundary $\partial_{+}SM$ parametrizes all geodesics going through $M$. The space of geodesics carries a natural measure (or volume form) which in turn equips $\partial_{+}SM$ with the measure
 \[d\mu(x,v):=|\langle \nu(x), v \rangle_g|dxdv\]
and we shall denote $L^{2}_{\mu}(\partial_{+}SM)$ the space of functions on $\partial_{+}SM$ with inner product
$$\langle u,w \rangle_{L^{2}_{\mu}(\partial_{+}SM)}=\displaystyle\int_{\partial_{+}SM}u w\,d\mu.$$
The measure $d\mu$ is natural in the following sense. If we consider the canonical map
\[\Phi:\{(x,v,t):\;(x,v)\in\partial_{+}(SM);\;t\in [0,\tau(x,v)]\}\to SM\]
given by $\Phi(x,v,t)=(\gamma_{(x,v)}(t),\dot{\gamma}_{(x,v)}(t))$ (the geodesic flow) then a calculation shows
that
\begin{equation}
\Phi^*(\Theta)=|\langle \nu(x), v \rangle_g|\,dxdvdt
\label{eq:justmu}
\end{equation}
where $\Theta$ is the canonical volume form of $SM$ (also called Liouville form in classical mechanics)
and $\Phi^*(\Theta)$ is a new volume form obtained by pulling back $\Theta$ via $\Phi$.

   % {\color{red}explain $\Phi^*$ and perhaps the whole thing in words?}. 

It is not hard to prove that $I$ extends as a bounded linear map \cite[Theorem 4.2.1]{Sharafudtinov1994}
\[I:L^{2}(M)\to L^{2}_{\mu}(\partial_{+}SM)\]
and hence we have a well defined adjoint $I^*:L^{2}_{\mu}(\partial_{+}SM) \to L^{2}(M)$ that can be easily computed using \eqref{eq:justmu}.
Explicitly
\[I^*w(x)=\int_{S_{x}M}w^{\sharp}(x,v)\,dv,\]
where $w^{\sharp}(x,v):=w(\gamma_{x,v}(-\tau(x,-v)),\dot{\gamma}_{x,v}(-\tau(x,-v)))$ and $S_{x}M$ denotes the set of unit tangent vectors at $x$.
In the literature that discusses the standard Radon transform, this operator is usually referred to as {\it back-projection} operator and appears prominently in the celebrated {\it filtered back-projection formula} \cite{R17,N86} (see \cite{Ku} for an excellent recent presentation of the classical Radon transform).
We can now define the `information operator' $I^*I:L^{2}(M)\to L^{2}(M)$. 
%There is an integral formula for $N$ which can be derived from the expressions above:
%\begin{equation}
%Nf(x)=2\int_{S_{x}M}dv\int_{0}^{\tau(x,v)}f(\gamma_{x,v}(t))\,dt.
%\label{eq:N}
%\end{equation}

The third and final condition that we will impose on $M$ is that it is {\it free of conjugate points}. Intuitively, this means that beams of geodesics emanating from a point do not focus on or converge to another point (as it would happen for the geodesics on the sphere connecting south and north poles).  Equivalently, two points in $M$ are joined by a unique geodesic (note that $M$ non-trapping and with strictly convex boundary implies that $M$ is contractible \cite{To}). This property is fundamental for us since it implies that the information operator is an elliptic pseudo-differential operator of order $-1$. Manifolds satisfying the three conditions -- non-trapping, strict convexity of the boundary and absence of conjugate points -- are called {\it simple}.

The theory of the X-ray transform is well-developed in the case of simple manifolds. If one considers only non-trapping manifolds with strictly convex boundary but allows for conjugate points, the operator $I^*I$ loses its pseudo-differential character. Strict convexity of the boundary is seen as less essential, but dropping it causes technical complications mostly arising from the non-continuity of the exit time $\tau$.

The discussion above extends without difficulties to the attenuated case. The {\it attenuated geodesic X-ray transform} $I_a f$ of a function $f \in C^{\infty}(M)$ with attenuation coefficient $a \in C^{\infty}(M)$ can be defined as the integral:
$$
I_a f(x,v) := \int_0^{\tau(x,v)} f(\gamma_{(x,v)}(t)) \text{exp}\left[ \int_0^t a(\gamma_{(x,v)}(s)) \,ds \right] dt,\;\, (x,v) \in \partial_+ SM.
$$
The transform $I_{a}$ extends as a bounded operator $I_{a}:L^{2}(M)\to L^{2}_{\mu}(\partial_{+}SM)$ with adjoint $I^*_{a}:L^2_{\mu}(\partial_{+}SM)\to L^{2}(M)$. In the case of simple manifolds, the information operator $I_{a}^{*}I_{a}$ displays the same features as $I^*I$.

We will consider noisy observations $Y$ of the X-ray transform $I_a f$ of an unknown function $f$. If $\mathbb W$ is a standard Gaussian white noise in the Hilbert space $L^2_\mu(\partial_+SM)$ and $\varepsilon>0$ a noise level, our data is
\begin{equation} \label{model0}
Y = I_af + \varepsilon \mathbb W.
\end{equation}
Up to a discretisation step described in the next section and the usual `Gaussianisation' of Poisson count data, this is a realistic approximate noise model for physical X-ray transform measurements. Assuming this model the following properties of the information operator $I_a^*I_a$ and its inverse will be crucial for the theory that follows: They imply that the inverse Fisher information exists for a variety of semi-parametric inference problems. Their proofs using techniques from microlocal analysis are given in Section \ref{pxray} below.

\begin{Theorem}\label{main0} Let $M$ be a simple manifold and suppose 
$$I_{a}:C^{\infty}(M)\to C^{\infty}(\partial_{+}SM)$$
is injective.
Let $d_M$ be any $C^\infty$ function that equals (the Riemannian) $dist(\cdot,\partial M)$ near the boundary and is positive on the interior of $M$.

a) The information operator $I^*_aI_a$ defines a bijection between $\{d_M^{-1/2} g: g \in C^\infty(M)\}$ and $C^\infty(M)$ and hence has a well defined inverse $$(I_a^*I_a)^{-1}:C^\infty(M) \to \{d_M^{-1/2} g: g \in C^\infty(M)\}$$ such that $I_a^*I_a (I_a^*I_a)^{-1} \psi = \psi$ for all $\psi \in C^\infty (M)$.

b) We have for some constant $c>0$ that depends only on $d,M$ $$\|I_a(d_M^{-1/2}h)\|_{L^2_\mu(\partial_+ SM)} \le c \|h\|_\infty$$
for every $h \in C(M)$.

c) For any $\psi \in C^\infty(M)$ we have that $I_a(I_a^*I_a)^{-1}\psi \in L^2_\mu(\partial_+ SM)$ and  
\begin{equation} \label{infobd}
\|I_a(I_a^*I_a)^{-1}\psi\|^2_{L^2_\mu(\partial_+(SM))}<\infty
\end{equation}
is the Cram\'er-Rao lower bound (inverse Fisher information) for estimation of the parameter $\langle f, \psi \rangle_{L^2(M)}$ in the model (\ref{model0}). 
\end{Theorem}

Injectivity of $I_{a}$ for simple manifolds is known in virtually all cases, so assuming it in the theorem is not a serious restriction. When $a=0$, injectivity of $I$ is a classical landmark result due to Mukhometov \cite{Mu}. In two dimensions, injectivity of $I_{a}$ is known in general \cite{SaU} and in dimensions $\geq 3$, $I_{a}$ is injective as long as $M$ admits a strictly convex function \cite{UV,PSUZ16}.

Given parts a) and b) the proof of the first assertion in Theorem \ref{main0}c is straightforward. The second assertion in Part c) then follows from standard semi-parametric theory (Chapter 25 in \cite{vdV98}):  An application of Lemma \ref{LAN} below implies that the model (\ref{model0}) is locally (asymptotically) normal (LAN) with LAN-norm $\|\cdot\|_{LAN}=\|I_a(\cdot)\|_{L^2_\mu(\partial_+ SM)}$, and since we have for all $h \in L^2(M)$ that
\begin{equation}\label{ch25}
\langle I_a h, I_a \tilde \psi \rangle_{L^2_\mu(\partial_+ SM)} = \langle h, \psi \rangle_{L^2(M)},~~~ \tilde \psi = (I_a^*I_a)^{-1}\psi,
\end{equation}
we can argue as in Section 7.4 in \cite{N17} to deduce the information lower bound $\|\tilde \psi\|_{LAN}^2$ from (\ref{infobd}). This identifies in particular the (lower bound for the) asymptotic minimax constant 
\begin{equation}\label{crmin}
\liminf_{\varepsilon \to 0} \inf_{\hat \psi} \sup_{f} \varepsilon^{-2} E_f(\hat \psi - \langle f,\psi\rangle)^2 \ge \|I_a(I_a^*I_a)^{-1}\psi\|^2_{L^2_\mu(\partial_+(SM))}
\end{equation}
where the infimum is taken over all estimators $\hat \psi=\hat \psi(Y)$ of $\langle f, \psi \rangle$ based on observations in the model (\ref{model0}), and where the supremum in $f$ extends over arbitrary $L^2$-neighbourhoods of $f_0$ of diameter $\varepsilon$.

\subsection{Bayesian Inference with Gaussian priors} \label{method}

We now address the statistical problem of recovering $f$ from a noisy observation of the X-ray transform $I_a f$, and propose numerical illustrations of the feasibility of the approach to general geometries. In what follows, we will take $M = \{(x_1,x_2)\in \Rm^2, x_1^2 + x_2^2 \le 1\}$, endowed with either the Euclidean metric $g_e = dx_1^2 + dx_2^2$ (generating the classical Radon transform), or the metric 
\begin{align}
    \begin{split}
	g(x_1,x_2) &= e^{2\lambda(x_1,x_2)} (dx_1^2 + dx_2^2), \\ 
	\lambda(x_1,x_2) &:= 0.45 (e^{-8((x_1-0.3)^2 + (x_2-0.3)^2)} - e^{-8((x_1+0.3)^2 + (x_2+0.3)^2)}),
    \end{split}    
    \label{eq:metric}
\end{align}
see Fig. \ref{fig:setting}. We will concentrate on the unattenuated case $a=0$ for conciseness. 
%Such a metric has a ``low sound speed'' region centered at $(0.3,0.3)$, generating a focusing effect on the geodesics, and a ``high sound speed'' region centered at $(-0.3,-0.3)$, creating a defocusing effect on the geodesics. 
We parameterise $\partial_+ SM$ using {\em fan-beam coordinates}, defined for $(\beta,\alpha)\in [0,2\pi)\times (-\pi/2,\pi/2)$ by 
\begin{align*}
    (\beta,\alpha) \mapsto \left( x = \left( \begin{smallmatrix} \cos\beta \\ \sin\beta \end{smallmatrix} \right), v = e^{-\lambda(x)} \left( \begin{smallmatrix} \cos(\beta+\pi+\alpha) \\ \sin(\beta+\pi+\alpha) \end{smallmatrix} \right) \right)\in \partial_+ SM,
\end{align*}
%\begin{align*}
%    \partial_+ SM = \left\{ x = \left( \begin{smallmatrix} \cos\beta \\ \sin\beta \end{smallmatrix} \right),\ v = e^{-\lambda(x,y)} \left( \begin{smallmatrix} \cos(\beta+\pi+\alpha) \\ \sin(\beta+\pi+\alpha) \end{smallmatrix} \right), \quad \beta\in [0,2\pi), \quad \alpha\in \left( \frac{-\pi}{2}, \frac{\pi}{2} \right)\right\}, 
%\end{align*}
with area element $d\mu(x,v) = \cos\alpha\ d\alpha\ d\beta$.

\medskip
\noindent {\bf Discretisation.} We assume in pratice that we are given noisy data at geodesics $\{\gamma_i\}_{i=1}^n$ emanating from a fan-beam sample $\{(\beta_i,\alpha_i)\}_{i=1}^n$, and that the unknown function is modelled as a finite sum $f = \sum_{j=1}^m f_j\phi_j$. Specifically, the domain is a triangular mesh with $m$ nodes $x_1, \dots, x_m$ (see Fig. \ref{fig:setting}, left), so that $f_j$ represents $f(x_j)$ and $\phi_j$ is a piecewise linear function on the mesh, uniquely defined by the relation $\phi_j(x_k) = \delta_{jk}$. We then seek to reconstruct $X = (f_1,\dots,f_m)^T$ from the observation
\begin{align} \label{model}
    Y = AX + \varepsilon W^{(n)}, \qquad Y = (y_1,\dots,y_n)^T,
\end{align}
where the discretised forward operator $A$ is an $n\times m$ matrix with entries $A_{ij} = I \phi_j (\beta_i,\alpha_i)$, and $W^{(n)}$ is Gaussian white noise on an {\it ad hoc} $n$-dimensional subspace of $L^2_\mu (\partial_+ SM)$. By {\it ad hoc} we mean that for this problem to be a faithful discretisation of the continuous one \eqref{model0}, one must endow the domain and range of $A:\Rm^m\to \Rm^n$ with inner products (described by matrices ${\mathsf m}$ and ${\mathsf n}$, respectively) which mimick the continuous inner products on $L^2(M)$ and $L^2(\partial_+ SM)$: %as an operator from $(\Rm^m,{\mathsf m})$ to $(\Rm^n,{\mathsf n})$, where ${\mathsf m}$ and ${\mathsf n}$ are matrices giving inner products which mimick the continuous ones. 
More precisely, if $f = \sum_j f_j \phi_j$ and $f' = \sum_j f'_j \phi_j$, then 
\begin{align*}
    \int_M f(x) f'(x)\ dx = \sum_{i,j = 1}^m {\mathsf m}_{ij} f_i f'_j, \qquad {\mathsf m}_{ij} := \int_M \phi_i(x) \phi_j(x)\ dx.
\end{align*}
Similarly, assuming here that the data comes from a uniform cartesian discretisation of $\partial_+ SM$ of size $n = n_\beta n_\alpha$, the $n$-dimensional subspace of $L^2_\mu (\partial_+ SM)$ on which \eqref{model} is posed has an orthogonal basis $\{e_i\}_{i=1}^n$, where $e_i$ equals $1$ on a pixel of dimensions $\frac{2\pi}{n_\beta}\times\frac{\pi}{n_\alpha}$ centered at $(\beta_i,\alpha_i)$ and $0$ elsewhere. A data sample $Y = (y_1,\dots,y_n)^T$ can then be viewed as a function $y = \sum_{i} y_i e_i \in L^2_\mu(\partial_+ SM)$, so that
\begin{align*}
    \int_{\partial_+ SM} y(\beta,\alpha) &y'(\beta,\alpha)\ \cos\alpha\ d\alpha\ d\beta = \sum_{i,j = 1}^n {\mathsf n}_{ij} y_i y'_j, \qquad \text{where} \\
    {\mathsf n}_{ij} &= \int_{\partial_+ SM} e_i(\beta,\alpha) e_j(\beta,\alpha)\ \cos\alpha\ d\alpha\ d\beta \approx \delta_{ij} \frac{2\pi^2}{n_\beta n_\alpha} \cos \alpha_i. 
\end{align*}
In particular, $\{e_i / \sqrt{ {\mathsf n}_{ii}}\}_{i=1}^n$ is an orthonormal family, and a realisation of a standard Gaussian white noise on this basis takes the form $\sum_{i=1}^n w_i (e_i/\sqrt{{\mathsf n}_{ii}})$ with $w_i\sim {\mathcal N}(0,1)$. In short, i.i.d. standard Gaussian noise in the `ad hoc' space used for discretisation equals $W^{(n)} \sim {\mathcal N}(0,{\mathsf n}^{-1})$ in usual coordinates. The computation of the forward matrix $A$ is now done by solving geodesics by ODEs first, then computing the integrals via Riemann sums, as in \cite{M14}. The main difference is that here $f$ is defined on an unstructured triangular mesh generated via the package \cite{Se}, and interpolating these values at any point is done using barycentric coordinates.

\begin{figure}[htpb]
    \centering
    \includegraphics[trim= 60 10 50 10, clip, height=0.19\textheight]{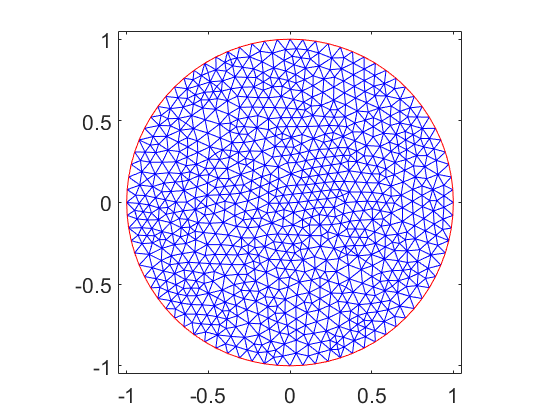}
    \includegraphics[trim= 50 10 50 10, clip, height=0.19\textheight]{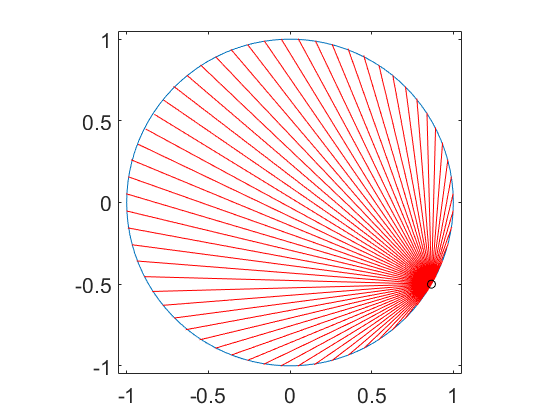}
    \includegraphics[trim= 30 10 40 10, clip, height=0.19\textheight]{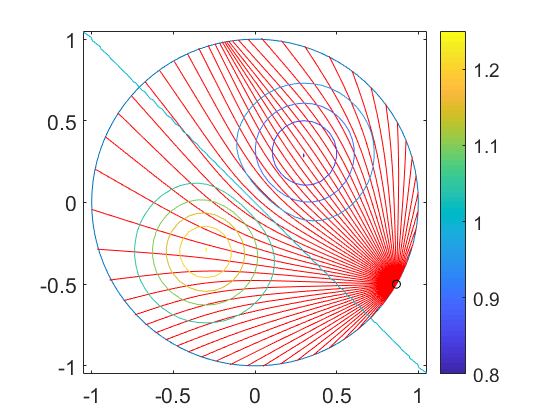}    
    \caption{Left to right: example of a mesh with 886 nodes; geodesics for the Euclidean geometry; geodesics for the metric given in \eqref{eq:metric}, superimposed to a contour plot of the ``sound speed'' $e^{-\lambda}$.}
    \label{fig:setting}
\end{figure}

% \medskip
\noindent{\bf Approach.} We take a basic Bayesian approach to this inverse problem: given a Gaussian prior $\Pi$ on the function $f$ (or its coefficient vector $X$), we assume the $Y_i|f$'s are generated from model \eqref{model} conditional on $f$, and obtain the posterior distribution on $f|Y$ by an application of Bayes' rule. %(and the measurability of the map $I_a$, as in \cite[Section 7.2]{N17}). 
For inference one needs to be able to calculate the posterior distribution, at least approximately. To this end, with the notation above, since $W^{(n)} \sim {\cal N} (0,{\mathsf n}^{-1})$, then $Y|X \sim {\cal N} (AX,\varepsilon^2 {\mathsf n}^{-1})$. Assuming the prior distribution is of the form $X \sim \N (0,\sigma^{-1} \Gamma)$ where the prior covariance matrix $\Gamma$ and the precision parameter $\sigma$ are known, a standard calculation gives the posterior distribution $X|Y \sim \N (X_c, \Gamma_c)$, where 
\begin{align}
    \Gamma_c := (\varepsilon^{-2} A^T {\mathsf n} A + \sigma\Gamma^{-1})^{-1}, \qquad X_c := \varepsilon^{-2} \Gamma_c A^T {\mathsf n} Y,
    \label{eq:posterior}
\end{align}
where $A^T$ denotes the standard matrix transpose. As the posterior distribution is Gaussian, the posterior mean equals the posterior mode $X_c$ (or MAP-estimate) and thus the Tikhonov-regulariser, see \cite{DLSV13} or \cite{GvdV17}. As a consequence the centre of mass of the posterior distribution is an approximation to the solution of the optimisation problem 
$$\min_{f \in V_\Pi} \left[ \varepsilon^{-2} \|Y-I_af\|^2_{L^2_\mu (\partial_+ SM)} + \sigma \|f\|_{V_\Pi}^2\right],$$ 
discretised into 
$$\min_{X \in \Rm^m} \left[ \varepsilon^{-2}  (Y-AX)^T {\mathsf n} (Y-AX) + \sigma X^T \Gamma^{-1} X\right],$$ 
where $V_\Pi \subset L^2(M)$ is the reproducing kernel Hilbert space (RKHS) of $\Pi$. [See \cite{GN16, GvdV17} for standard properties of Gaussian processes and their RKHS.] Natural choices for $V_\Pi$ are those coming from kernel-type Gaussian process whose covariance is prescribed by a fixed positive definite function $K(\cdot, \cdot)$, see also Remark \ref{sobp} below. In particular, we choose here the Mat\'ern kernel $K(x_i,x_j) = k_{\nu,\ell}(|x_i-x_j|)$, where 
\begin{align*}
    k_{\nu,\ell} (r) := \frac{2^{1-\nu}}{\Gamma(\nu)} \left( \frac{\sqrt{2\nu}r}{\ell} \right)^\nu K_\nu \left(\frac{\sqrt{2\nu}r}{\ell}\right),
\end{align*}
and where $K_\nu$ denotes the modified Bessel function of the second kind. In the examples below, the four parameters $(\varepsilon, \sigma, \nu, \ell)$ are assumed to be known. To address uncertainty on these parameters, hierarchical models can be considered and efficient methods can be derived to compute features of the posterior distribution, see, e.g., the recent article \cite{BSV}. 

\medskip
\noindent {\bf Experiments.} The phantoms used are given Figure \ref{fig:phantoms}, $f_1$ is the so-called `modified Shepp-Logan' phantom (compactly supported) and $f_2 = h_2/ \sqrt{d_M}$ with $h_2 \in C^\infty(\overline{M})$ and $d_M(x,y) := \frac{1}{2} (1-x^2-y^2)$ (as discussed in the next section, the scaling by $\sqrt {d_M}$ is natural in this inverse problem). In all examples, the mesh has $m=6027$ nodes and we use $n = 14450$ geodesics. The other parameters are given by $\varepsilon = 10^{-3}$, $\sigma=1$, $\nu= 1.5$ and $\ell = 0.2$. Sampling the posterior distribution is done by drawing $X = X_c + GZ$, where $Z\sim \N(0,I_{m\times m})$ and $G$ is a matrix satisfying $GG^T = \Gamma_c$ (defined in \eqref{eq:posterior}), obtained for instance by Cholesky decomposition (here one may notice that this step is a much cheaper option than computing the SVD of the information operator). To compute forward data, we use the code \cite{M14} which allows to produce 'true' data with higher accuracy, thereby avoiding the {\it inverse crime} of using a forward and inverse solver on the same computational grid. 

\begin{figure}[htpb]
    \centering
    \includegraphics[trim= 30 10 30 10, clip, height=0.20\textheight]{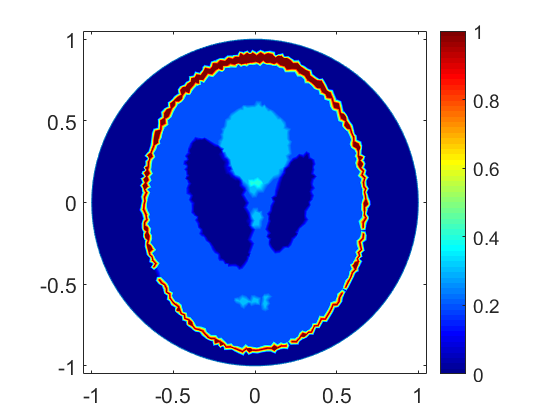} \ \ 
    \includegraphics[trim= 30 10 30 10, clip, height=0.20\textheight]{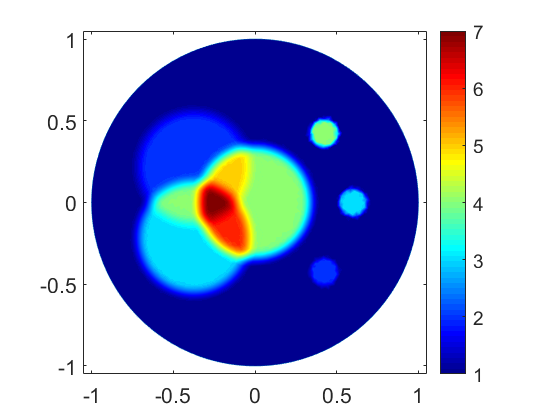}
    \caption{Left: $f_1$, the Shepp-Logan phantom (compactly supported). Right: the function $h_2$ such that $f_2 :=  h_2/\sqrt{d_M}$ blows up at the boundary.}
    \label{fig:phantoms}
\end{figure}

\medskip
\noindent {\bf Example 1.} Euclidean geometry, reconstruction of $f_1$ from its Euclidean ray transform. We compute the posterior distribution and visualise the mean and sample draws. Results are visualised Fig. \ref{fig:Ex1}. As the Shepp-Logan phantom has spatial variations which may be too sharp to be captured by the prior, we expect over-smoothed reconstructions near sharp edges. This can be seen, e.g., on Fig. \ref{fig:Ex1} (bottom-right). 

\begin{figure}[htpb]
    \centering
    \includegraphics[trim= 20 10 20 10, clip, height=0.20\textheight]{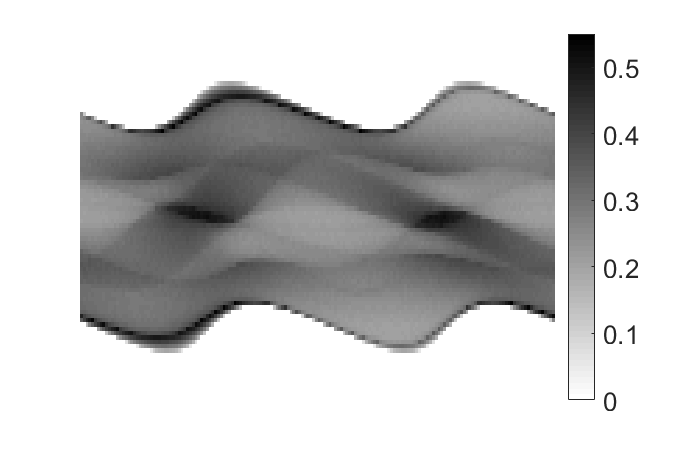}
    \includegraphics[trim= 20 10 20 10, clip, height=0.20\textheight]{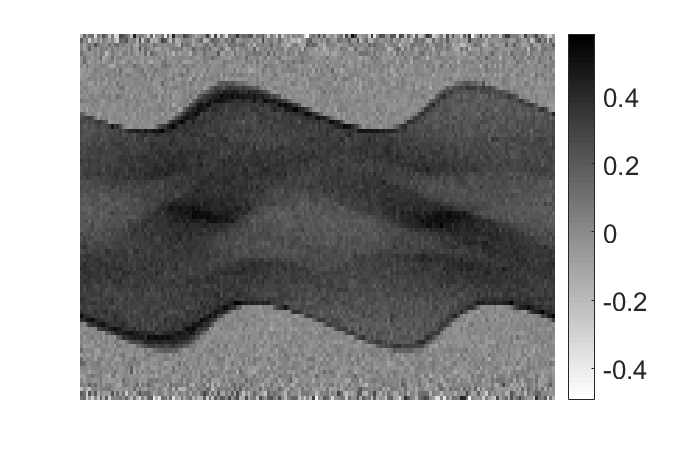} \\
    \includegraphics[height=0.20\textheight]{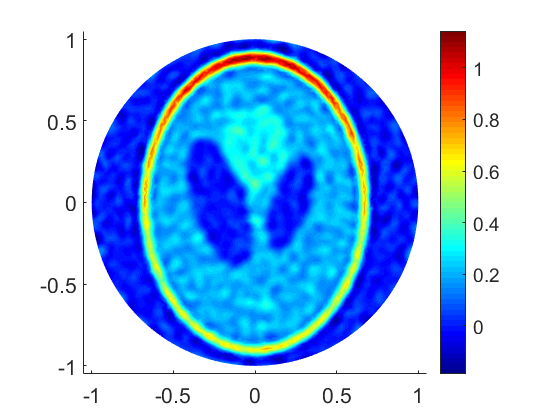}
    \includegraphics[height=0.20\textheight]{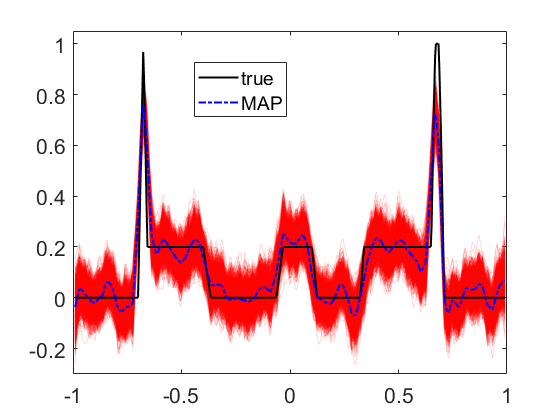}
    \caption{Example 1. Left to right. Top row: $If_1$; $If_1$ noisy (with $\beta$ on the horizontal axis and $\alpha$ on the vertical axis). Bottom row: posterior mean; cross-section on $\{x_2 = 0\}$ of 2000 posterior samples.}
    \label{fig:Ex1}
\end{figure}

\medskip
\noindent {\bf Example 2.} Same as Example 1, except that the geometry is the non-Euclidean one characterized by the metric in \eqref{eq:metric} and geodesics displayed Fig. \ref{fig:setting} (right). Results are displayed Fig. \ref{fig:Ex2}, illustrating the applicability of the approach to non-standard geometries.

\begin{figure}[htpb]
    \centering
    \includegraphics[trim= 20 10 20 10, clip, height=0.20\textheight]{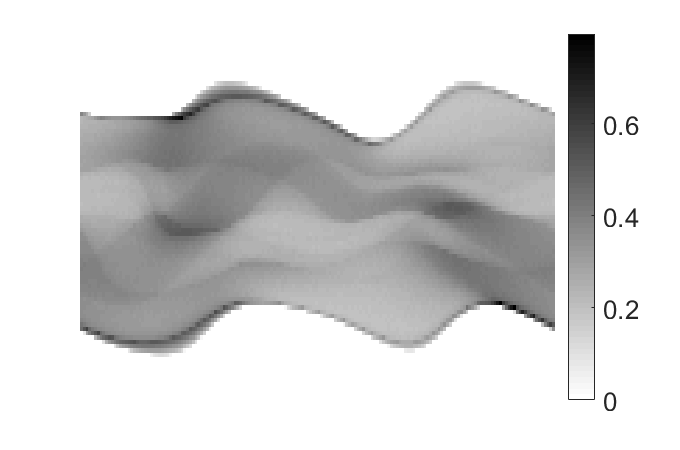}
    \includegraphics[trim= 20 10 20 10, clip, height=0.20\textheight]{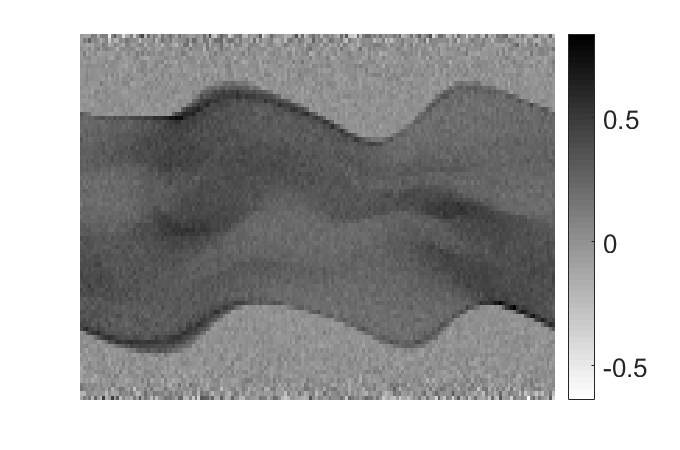} \\
    \includegraphics[height=0.20\textheight]{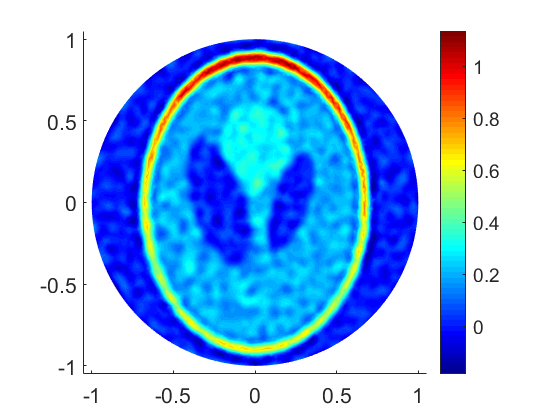}
    \includegraphics[height=0.20\textheight]{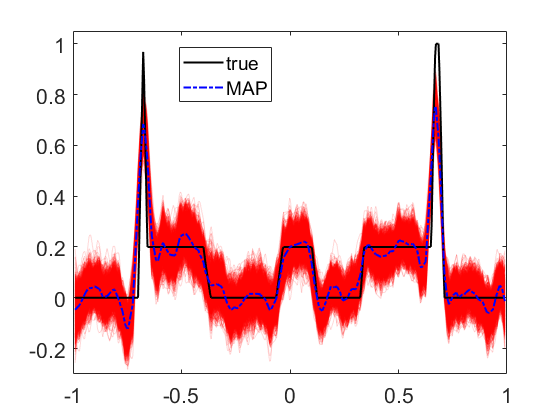}
    \caption{Example 2. Left to right. Top row: $If_1$; $If_1$ noisy (with $\beta$ on the horizontal axis and $\alpha$ on the vertical axis). Bottom row: posterior mean; cross-section on $\{x_2 = 0\}$ of 2000 posterior samples.}
    \label{fig:Ex2}
\end{figure}

\medskip
\noindent {\bf Example 3.} Reconstruction of $f_2$ in Euclidean geometry, with noise level $\varepsilon = 10^{-2}$ (all other parameters unchanged). As explained in the theory that follows in the next section, an appropriate prior for $f_2$ should be of the form $d_M^{-1/2} h_2$ where $h_2$ is drawn from a `standard' Gaussian prior modelling a regular function. For numerical purposes, it should be more stable to work with $h_2$, and try to reconstruct $h_2$ from the transform $I_d h_2 := I(d_M^{-1/2} h_2)$, as the transform $I_d$ naturally compensates for the blowup by integrating. In the implementation, the only change is to work with the discretised version of $I_d$ rather than $I$ (call the corresponding matrix $A_d$), everything else being kept equal. As may be observed on the middle row of Fig. \ref{fig:Ex3}, the reconstruction of $h_2$ is quite robust, especially at the boundary despite the blowup of $f_2$ there. For comparison, the bottom row of Fig. \ref{fig:Ex3} gives the outcome of just inverting for $f_2$ using $A$ with the usual prior on $f_2$ instead of $h_2$ (as in Example 1). As expected, the latter approach is manifestly more unstable near the boundary, and this instability is propagated to the reconstruction in the interior of $M$, as Fig.~\ref{fig:Ex3} illustrates. %See Remark \ref{sobp} for more discussion of these boundary effects. %The deeper reason behind is the following: the information operator $I^*I$ is a non-local pseudo-differential operator, hence prior mis-specification of the boundary behaviour of $f$ can propagate into invalid posterior inference for $f$ also in the interior of $M$. As a consequence, unless it is known {\it a priori} that $f$ is supported in the interior of $M$, the re-weighting of the initial prior by $d_M^{-1/2}$ is crucial for efficient reconstruction of the entire image.

\begin{figure}[htpb]
    \centering
    \includegraphics[trim= 20 10 20 10, clip, height=0.21\textheight]{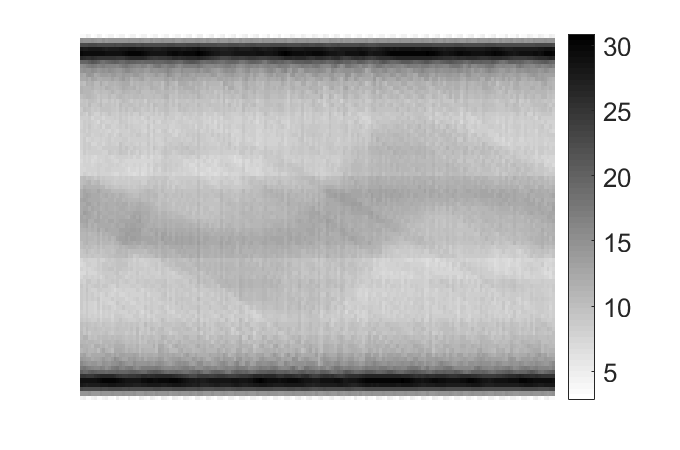}
    \includegraphics[trim= 20 10 20 10, clip, height=0.21\textheight]{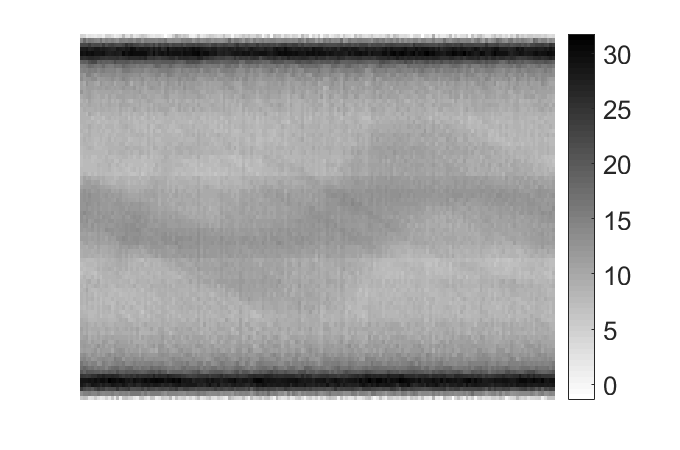} \\
    \includegraphics[height=0.24\textheight]{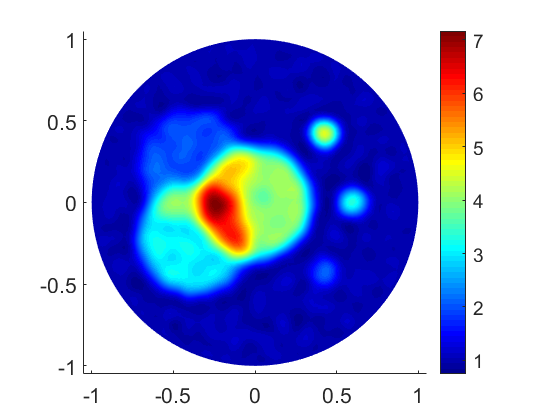}
    \includegraphics[height=0.24\textheight]{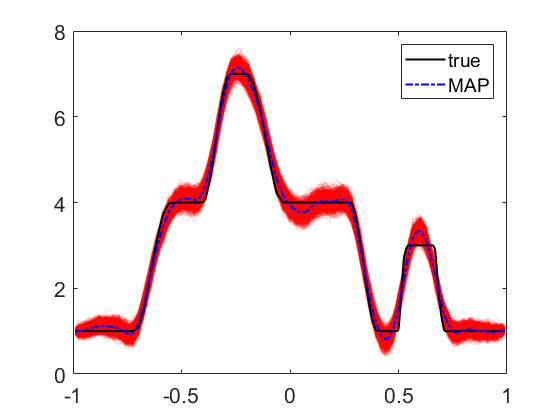} \\
    \includegraphics[height=0.24\textheight]{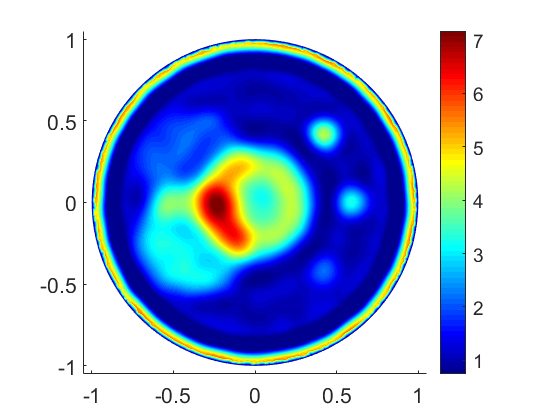}
    \includegraphics[height=0.24\textheight]{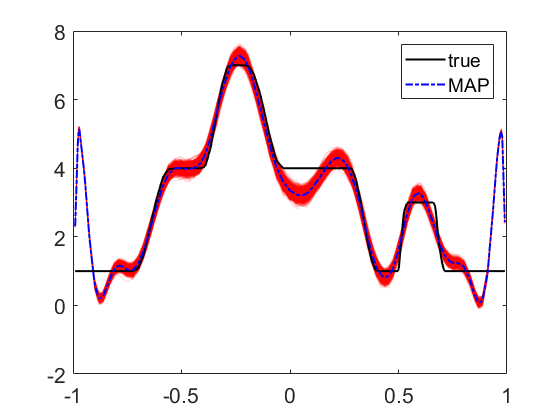}
    \caption{\small Example 3. Left to right. Top row: $If_2$; $If_2$ noisy. Middle row: posterior mean and 2000 cross-sectioned posterior samples for $h_2$. Bottom row: posterior mean and 2000 cross-sectioned posterior samples for $f_2$, divided by $d_M^{-1/2}$ for comparison with middle row.}
    \label{fig:Ex3}
\end{figure}

\subsection{Asymptotic normality of the posterior distribution and of the MAP estimator} \label{bvmsec}

The statistical methodology laid out in the previous section did not rely on any problem-specific regularisation (it just was based on a standard Gaussian process and the penalty norm of its RKHS), particularly no computation of the $SVD$ of the forward operator $I_a$ was required. One may wonder whether the Bayes solution of this inverse problem can be objectively trusted in the sense that it solves the inverse problem in a prior-independent way. We will now show that for the recovery of arbitrary $C^\infty$-aspects of $f$, posterior based inference is not only valid but actually optimal from an information theoretic point of view. The theory will be given in the `continuous' Gaussian white noise model (\ref{model0}).

We start with a Gaussian Borel probability measure $\tilde \Pi$ supported in the space $C(M)$ of bounded continuous functions on $M$. If $h \sim \tilde \Pi$ we let the prior $\Pi$ for $f$ be the law of the random function 
\begin{equation} \label{transf}
f(x)=D(h)(x) :=h(x) /\sqrt{d_M(x)},~x \in M,
\end{equation}
where $d_M$ is any function as in Theorem \ref{main0}. By standard arguments (Exercise 2.6.5 in \cite{GN16} or Lemma I.16 in \cite{GvdV17}), if $V_{\tilde \Pi}$ is the RKHS of the initial Gaussian measure $\tilde \Pi$ then the RKHS $V_\Pi$ of the induced prior has norm $\|\cdot\|_{V_\Pi} = \|\sqrt {d_M}(\cdot)\|_{V_{\tilde \Pi}}$. The linear mapping $D$ transforms a standard Gaussian prior into one that allows for singularities of functions at the boundary $\partial M$ of a form suggested by Theorem \ref{main0}.

We will now give some precise asymptotic ($\varepsilon \to 0$) results about the statistical behaviour of the posterior distribution arising from such a prior, under the frequentist assumption that a fixed $f_0$ generates the observations in (\ref{model0}). We will require a mild condition on the prior and on $f_0$ expressed through the concentration function of the initial probability measure $\tilde \Pi$
\begin{equation} \label{concsup}
\phi_{\tilde \Pi, f_0}(\delta) = \inf_{v \in V_{\tilde \Pi}, \|v-\sqrt {d_M} f_0\|_{\infty} \le \delta} \left[\frac{\|v\|_{V_{\tilde \Pi}}^2}{2} - \log \tilde \Pi (h:\|h\|_{\infty} \le \delta) \right],
\end{equation}
which characterises the asymptotics of the small ball probabilities $\tilde \Pi(h: \|h- {\sqrt {d_M}} f_0\|_\infty \le \delta)$ of $\tilde \Pi$ as $\delta \to 0$. The concentration function of Gaussian priors is well studied see \cite{vdVvZ08} or also Chapter 2.6 in \cite{GN16} and \cite{GvdV17}, and the condition that follows is mild -- it can be shown to be satisfied for all sufficiently rich Gaussian processes arising from positive definite kernels $K$, as soon as $\sqrt {d_M} f_0$ satisfies standard smoothness conditions, see Remark \ref{sobp}.

\begin{condition} \label{priorc}
Let $f_0: M \to \mathbb R$ such that $\sqrt {d_M}f_0 \in C(M)$. Let $\tilde \Pi$ be a Gaussian Borel probability measure on $C(M)$ whose RKHS $V_{\tilde \Pi}$ contains $C^\infty(M)$ and whose concentration function satisfies, for $c$ equal the constant from Theorem \ref{main0}b and for some sequence $\delta_\varepsilon  \to 0$ such that $\delta_\varepsilon/\varepsilon \to \infty$,
\begin{equation} \label{cfcond}
\phi_{\tilde \Pi, f_0}(\delta_\varepsilon/2c) \le (\delta_\varepsilon/\varepsilon)^2.
\end{equation}
\end{condition}

\begin{condition} \label{setting}
Let $P_{f_0}^Y$ be the law generating the equation $Y=I_af_0+\varepsilon \mathbb W$, where $I_a: L^2(M) \to L^2_\mu(\partial_+(SM))$ is the X-ray transform under the conditions of Theorem \ref{main0}, $\mathbb W$ is a white noise in $L^2_\mu(\partial_+(SM))$, and $\varepsilon>0$ is a noise level. Let $\Pi(\cdot|Y)$ be the posterior distribution arising from observing (\ref{model0}) under prior $\Pi=\mathcal L(f)$, where $\mathcal L(f)=\mathcal L(d_M^{-1/2}h), h \sim \tilde \Pi$, with $\tilde \Pi$ satisfying Condition \ref{priorc} for the given $f_0$. 
\end{condition}

Our main statistical result is the following Bernstein-von Mises theorem for posterior inference on $\langle f, \psi \rangle_{L^2(M)}$ for arbitrary test functions $\psi \in C^\infty(M)$. The idea of its proof is partly inspired by \cite{CN13, CN14, C14, CR15, C17}, where however priors have to be used that are diagonal in the inner product induced by the information operator. This is not the case in the inverse problem setting we consider here, but the invertibility result in Theorem \ref{main0} combined with an adaptation of ideas in \cite{CN13, CN14} allow to overcome this difficulty. We give the result for smooth $\psi$ but our techniques can be used to obtain results for less regular $\psi$ as well in principle, see Remark \ref{more}.

We employ the usual notion of weak convergence of laws $\mathcal L(X_n) \to^\mathcal L \mathcal L(X)$ of real random variables $X_n, X$ that converge in distribution, $X_n \to^d X$. In (\ref{bvm}) below we claim convergence of \textit{random} laws $\mu_n \to^\mathcal L \mu$ in probability, which means that for $\beta$ any metric for weak convergence of laws (11.3 in \cite{D02}),  the real random variables $\beta\left(\mu_n, \mu \right)$ converge to zero in probability.

\begin{Theorem} \label{main}
Assume Condition \ref{setting}. If $f \sim \Pi(\cdot|Y)$, then for every $\psi \in C^\infty(M)$ we have as $\varepsilon \to 0$ that
\begin{equation}\label{bvm}
\mathcal L\big(\varepsilon^{-1} \big(\langle f, \psi \rangle_{L^2(M)} - \hat \Psi \big) |Y\big) \to^\mathcal L \mathcal N(0, \|I_a(I_a^*I_a)^{-1}\psi\|^2_{L^2_\mu(\partial_+(SM))})
\end{equation}
 in $P_{f_0}^Y$-probability, where $$\hat \Psi = \langle f_0, \psi \rangle_{L^2(M)} - \varepsilon \langle I_a (I_a^*I_a)^{-1}\psi, \mathbb W\rangle_{L^2_\mu(\partial_+(SM))}.$$
 \end{Theorem}

 \begin{Remark}  [Examples of Gaussian priors and of $f_0$'s] \normalfont  \label{sobp}
Regarding $M$ as a subset of $\mathbb R^d$, most Gaussian processes that model regular functions in $C(\mathbb R^d)$, when restricted to $C(M)$, will satisfy Condition \ref{priorc}, if we assume that $\sqrt {d_M} f_0 $ is sufficiently regular. For example let $K: \mathbb R^d \to \mathbb R$ be a positive definite kernel function whose Fourier transform $FK$ satisfies, for all $\|u\|$ large enough,
\begin{equation} \label{decay}
c_0\|u\|^{-2s} \le FK(u) \le c_1 \|u\|^{-2s},~~c_0 < c_1,~~s>d/2,
\end{equation}
a concrete example being provided by the Mat\'ern kernel (see p.313 in \cite{GvdV17} and also Section 9.6 of \cite{AS64}). For such $K$ we can define a unique centred stationary Gaussian process $(G(x): x \in \mathbb R^d)$ with covariance $EG(x)G(y)=K(x-y), x,y \in \mathbb R^d$.  The Gaussian process $(G(x): x \in M)$ obtained by restriction to $M \subset \mathbb R^d$ defines a tight Gaussian measure $\gamma_K$ on $C(M)$, and its RKHS coincides with the standard Sobolev space $H^s(M)$ obtained from restricting elements of $H^s(\mathbb R^d)$ to $M$. Moreover if $f_0 = d_M^{-1/2}\phi_0$ for some $\phi_0 \in H^s(M), s>d/2,$ then $\tilde \Pi = \gamma_K$ satisfies Condition \ref{priorc} with $\delta_\varepsilon \approx \varepsilon^{2s/(2s+d)}$. Likewise, if $\phi_0$ is $\alpha$-H\"older continuous on $M$ for some $\alpha>0$ (including the case of arbitrary $f_0 \in C^\infty(M)$), it can be approximated from elements in $H^s(M)$ in $\|\cdot\|_\infty$-norm and a sequence $\delta_\varepsilon \to 0$ for which Condition \ref{priorc} holds can still be found. These facts can be proved just as in \cite{GvdV17}, p.330f. 
\end{Remark}

From the previous theorem we can deduce the asymptotic distribution of the posterior mean $E^\Pi[f|Y]$, which, since the posterior distribution is also a Gaussian measure, equals the posterior mode (MAP estimate). From Corollary 3.10 in \cite{DLSV13} (see also Section 11.7 in \cite{GvdV17}), MAP estimates can further be seen to equal the Tikhonov-regularisers with RKHS norm as penalty function.  Note that in our infinite-dimensional setting the Tikhonov regulariser is defined as the maximiser in $f$ of the Onsager-Machlup functional 
\begin{equation}
Q(f) = \frac{1}{\varepsilon^2} \langle I_af, Y \rangle_{L_\mu^2(\partial_+ SM)} - \frac{1}{2\varepsilon^2} \|I_af\|_{L_\mu^2(\partial_+ SM)}^2 - \frac{1}{2} \|f\|_{V_{\Pi}}^2.
\end{equation}
In the discrete setting from Section \ref{method} this is equivalent to minimising $Q(f)=\frac{1}{\varepsilon^2} \|Y-I_a f\|^2 + \|f\|_{V_\Pi}^2$ as usual, but in our setting $Y \notin L_\mu^2(\partial_+ SM)$, so the preceding formulation is the appropriate one.

\begin{Theorem} \label{tikh}
Let $\bar f = \bar f(Y)= E^\Pi[f|Y] \in C(M)$ be the mean of the posterior distribution in Theorem \ref{main}. Then for every $\psi \in C^\infty(M)$ we have $\langle \bar f(Y), \psi \rangle_{L^2(M)} - \hat \Psi = o_{P_{f_0}^Y}(\varepsilon)$ as $\varepsilon \to 0$ and thus also, under $P_{f_0}^Y$,
\begin{equation} \label{meanbvm}
\frac{1}{\varepsilon}\langle \bar f - f_0, \psi \rangle_{L^2(M)} \to^d Z \sim \mathcal N(0, \|I_a(I_a^*I_a)^{-1}\psi\|^2_{L^2_\mu(\partial_+(SM))}).
\end{equation}
In particular in (\ref{bvm}) in Theorem \ref{main} we may replace the centring $\hat \Psi$ by $\langle \bar f, \psi \rangle_{L^2(M)}$.
\end{Theorem}

\begin{Remark}[Exact asymptotic minimaxity] \normalfont
The proof of the last theorem implies that convergence of all moments in (\ref{meanbvm}) occurs, and hence $\bar f$ attains the lower bound constant from (\ref{crmin}) in the small noise limit. Thus $\langle \bar f, \psi \rangle_{L^2(M)}$ is an asymptotically exact minimax estimator of $\langle f_0, \psi \rangle_{L^2(M)}$.
\end{Remark}

\begin{Remark}  [Confident credible sets] \normalfont
Theorems \ref{main} and \ref{tikh} justify the following construction of a confidence set for the Tikhonov regulariser: Consider a credible interval $$C_\varepsilon = \{x \in \mathbb R: |\langle \bar f, \psi \rangle - x| \le R_\varepsilon\},~~R_\varepsilon \text{ s.t. } \Pi(C_\varepsilon|Y)=1-\alpha,$$ for some given significance level $0<\alpha<1$. The frequentist coverage probability of $C_\varepsilon$ will satisfy (arguing as in the proof of Theorem 7.3.23 in \cite{GN16}) $$P^Y_{f_0}(\langle f_0, \psi \rangle \in C_\varepsilon) \to 1-\alpha,~~\text{and }\varepsilon^{-1} R_\varepsilon \to^{P_{f_0}^Y} \Phi^{-1}(1-\alpha)$$ as $\varepsilon \to 0$. Here $\Phi^{-1}$ is the continuous inverse of $\Phi = \Pr (|Z| \le \cdot)$ with $Z$ as in (\ref{meanbvm}). To implement this confidence set we use the posterior sampling method from Section \ref{method} to numerically approximate the quantile constants $R_\varepsilon$ -- computation of $Var(Z)$, which could be intricate, is not required.
\end{Remark}
\begin{Remark} [Extensions] \normalfont \label{more}
The above theorem shows that semi-parametrically efficient recovery of $C^\infty$ aspects of $f$ is possible. Following the program laid out in the papers \cite{CN13, CN14, C14, C17, N17} one could in principle proceed to use the estimates in the proof of Theorem \ref{main} to derive a result for posterior reconstruction of the \textit{entire} parameter $f$ in suitable norms via bounding $\varepsilon^{-1}\langle f-f_0, \psi \rangle|Y$ \textit{uniformly} in collections of functions $\psi$ of bounded Sobolev norms. The approximation theoretic arguments required to do that in the present setting involve delicate boundary issues, with standard Sobolev spaces as approximation scales having to be replaced by the H\"ormander spaces introduced in Section \ref{pxray} below. The execution of these arguments is possible but quite technical and beyond the scope of this paper.
\end{Remark}

\section{Proofs for Section \ref{bvmsec}}

\subsection{Proof of Theorem \ref{main}}

Let $H_i, i=1,2$ be separable Hilbert spaces and consider the equation
\begin{equation*} 
Y =  G(f) + \eps \mathbb W,~~~\eps>0,
\end{equation*}
where $G: H_1 \to  H_2$ is a Borel measurable mapping and $\mathbb W$ is a centred Gaussian white noise process $(\mathbb W(h): h \in H_2)$ with covariance $E \mathbb W(h) \mathbb W(g) = \langle h,g \rangle_{H_2}$. Observing $Y$ then means that we observe a realisation of the Gaussian process $(Y(h) = \langle Y, h \rangle_{H_2} : h \in H_2)$. We sometimes write $\langle \mathbb W, h \rangle_{H_2}$ for the random variable $\mathbb W(h)$. Arguing as in Section 7.3 in \cite{N17} the posterior distribution of $f|Y$ exists and equals 
\begin{equation} \label{posto}
\Pi(B|Y) = \frac{\int_B p_f(Y) d\Pi(f)}{\int_\mathcal F p_f(Y) d\Pi(f)},~~~ B \in \mathcal B_{H_1}~\text{a Borel set in } H_1,
\end{equation} 
where $p_f(Y)$ is a likelihood function with respect to a suitable dominating measure. The following result is a standard application of the Cameron-Martin theorem (see, eq.~(111) in \cite{N17}).
\begin{Lemma} \label{LAN}
Let $\ell(f) = \log p_f(Y)$ and assume $Y= G(f_0) + \eps \mathbb W$ for some fixed $f_0 \in H_1$. Then if $G$ is also linear, we have for any $f,g \in H_1$, $$\ell(f)-\ell(g)=-\frac{1}{2\varepsilon^2} \left(\| G(f-f_0)\|_{H_2}^2 - \|G(g-f_0)\|_{H_2}^2\right) + \frac{1}{\varepsilon} \langle G(f-g), \mathbb W \rangle_{H_2} $$
\end{Lemma}

We now prove Theorem \ref{main}, and will use the above lemma with $H_1 = L^2(M)$, $H_2 = L_\mu^2(\partial_+ SM)$, $G=I_a$. In what follows the total variation norm between finite  measures $\mu, \nu$ is defined to equal the supremum $\|\mu-\nu\|_{TV}:=\sup_{B}|\mu(B)-\nu(B)|$ over all Borel sets $B$.

\begin{Lemma} \label{postex}
Let $\tilde \Pi$ be a Gaussian Borel probability measure on $C(M)$, and for fixed $f_0$ assume its concentration function $\phi_{\tilde \Pi, f_0}$ satisfies (\ref{cfcond}) for some $\delta=\delta_\varepsilon \to 0$. Let $\Pi$ be the prior for $f$ corresponding to the law of $f=D(h), h \sim \tilde \Pi$ with $D$ as in (\ref{transf}), and let $\Pi(\cdot|Y)$ be the resulting posterior distribution arising from observing $Y=I_af+\varepsilon \mathbb W$, where $I_a$ is the X-ray transform from Theorem \ref{main0}. Then for any Borel set $D_\varepsilon \subset L^2(M)$ for which
\begin{equation}
\Pi(D_\varepsilon^c) \le e^{-D_0(\delta_\varepsilon/\varepsilon)^2}~~\text{for some } D_0>3
\end{equation}
and all $\varepsilon>0$ small enough, we have 
\begin{equation} \label{conlim}
\Pi(D_\varepsilon^c|Y) \to 0 \text{ and }\|\Pi^{D_\varepsilon}(\cdot|Y)-\Pi(\cdot|Y)\|_{TV}  \to 0
\end{equation}
as $\varepsilon \to 0$ in $P_{f_0}^Y$-probability. Here $\Pi^{D_\varepsilon}(\cdot|Y)$ is the posterior distribution arising from the prior $\Pi(\cdot \cap D_\varepsilon)/\Pi(D_\varepsilon)$ restricted to $D_\varepsilon$ and renormalised.
\end{Lemma}
\begin{proof}
It suffices to prove the first limit in (\ref{conlim}), the second then follows from the basic inequality $\|\Pi^{D_\varepsilon}(\cdot|Y)-\Pi(\cdot|Y)\|_{TV} \le 2 \Pi(D_\varepsilon^c|Y)$.

We have from (\ref{posto}) that $$\Pi(B|Y)= \frac{\int_B e^{\ell(f)-\ell(f_0)} d\Pi(f)}{\int_\mathcal F e^{\ell(f)-\ell(f_0)} d\Pi(f)},~~ B \in \mathcal B_{L^2(M)},$$
and under $P_{f_0}^Y$
can use Lemma \ref{LAN} to see $$\ell(f)-\ell(f_0) = -\frac{1}{2\varepsilon^2}\|I_a(f-f_0)\|_{H_2}^2 + \frac{1}{\varepsilon}\langle I_a(f-f_0), \mathbb W \rangle_{H_2}.$$ Let $\nu$ be any probability measure on the set  $B=\{f: \|I_a(f-f_0)\|_{H_2}^2 \le \delta^2\}$. For any $C>0$ we have from Jensen's inequality
\begin{align*}
& P_{f_0}^Y \left(\int_B e^{\ell(f)-\ell(f_0)} d\nu(f) \le e^{-(1+C) (\delta_\varepsilon/\varepsilon)^2} \right) \\
&\le \Pr \left(\int_B \big(-\frac{1}{2\varepsilon^2}\|I_a(f-f_0)\|_{H_2}^2 + \frac{1}{\varepsilon}\langle I_a(f-f_0), \mathbb W \rangle_{H_2}\big )  d\nu(f) \le -(1+C) \frac{\delta_\varepsilon^2}{\varepsilon^2} \right) \\
& \le  \Pr \left(\left|\int_B  \frac{1}{\varepsilon}\langle I_a(f-f_0), \mathbb W \rangle_{H_2} d\nu(f) \right| \ge C \frac{\delta_\varepsilon^2}{\varepsilon^2} \right) \le e^{-C^2 (\delta_\varepsilon/\varepsilon)^2/2}
\end{align*}
since the standard Gaussian tail bound $P(|Z-EZ|>u) \le e^{-u^2/2Var(Z)}$ applies to the random variable $Z=\int_B \varepsilon^{-1} \langle I_a(f-f_0), \mathbb W \rangle_{H_2} d\nu(f)$ which has a centred normal distribution with variance bounded, again using Jensen's inequality, by
\begin{align*}
E\left[\int_B  \varepsilon^{-1}\langle I_a(f-f_0), \mathbb W \rangle_{H_2} d\nu(f) \right]^2 & \le \varepsilon^{-2}\int_B E \langle I_a(f-f_0), \mathbb W\rangle_{H_2}^2 d\nu(f) \le  \frac{\delta^2}{\varepsilon^2},
\end{align*}
recalling that $\mathbb W$ is a centred Gaussian white noise in $H_2$. Now we choose $\nu = \Pi(\cdot \cap B)/\Pi(B)$ and let $$A_\varepsilon=\Big\{\int_B e^{\ell(f)-\ell(f_0)} d\nu(f) \le e^{-2(\delta_\varepsilon/\varepsilon)^2}\Big\},$$ for which $P_{f_0}^Y(A_\varepsilon) \le e^{- (\delta_\varepsilon/\varepsilon)^2/2} \to 0$ by what precedes (with $C=1$). For $E_{f_0}^Y$ the expectation operator corresponding to $P_{f_0}^Y$ and by Markov's inequality, it suffices to prove convergence to zero of
$$E_{f_0}^Y\Pi(D_\varepsilon^c|Y)= E_{f_0}^Y\Pi(D_\varepsilon^c|Y)1_{A_\varepsilon} + E_{f_0}^Y\Pi(D_\varepsilon^c|Y)1_{A^c_\varepsilon}.$$ Since $\Pi(\cdot|Y) \le 1$ the first quantity is less than $P_{f_0}^Y(A_\varepsilon)$ and hence converges to zero. For the second term we have
\begin{align} \label{lowbd}
 E_{f_0}^Y\Pi(D_\varepsilon^c|Y)1_{A^c_\varepsilon} &\le \frac{e^{2 (\delta_\varepsilon/\varepsilon)^2}}{\Pi(f: \|I_a(f-f_0)\|_{H_2}^2 \le \delta_\varepsilon^2)} \int_{D_\varepsilon^c} E_{f_0}^Y[e^{\ell(f)-\ell(f_0)}]d\Pi(f) \notag \\
& \le e^{2 (\delta_\varepsilon/\varepsilon)^2} e^{\phi_{\Pi, f_0}(\delta_\varepsilon/2)} \Pi(D_\varepsilon^c).
\end{align}
noting that $E_{f_0}^Y[e^{\ell(f)-\ell(f_0)}]=1$ and where
\begin{equation}\label{concinf}
\phi_{\Pi, f_0}(\delta) = \inf_{w \in V_\Pi, \|I_a(w-f_0)\|_{L^2_\mu(\partial_+ SM)} \le \delta} \Big[\frac{\|w\|_{V_\Pi}^2}{2} - \log \Pi (f:\|I_af\|_{L^2_\mu(\partial_+ SM)} \le \delta) \Big],
\end{equation}
using Proposition 2.6.19 and Exercise 2.6.5 in \cite{GN16}, with RKHS $I_a(V_\Pi)$ of $I_af$ isometric to $V_\Pi$ since $I_a$ is linear and injective. Now we have for all $\delta>0$ that $\phi_{\Pi, f_0} (\delta) \le  \phi_{\tilde \Pi, f_0}(\delta/c)$ since Theorem \ref{main0} implies
\begin{equation*}
\|I_af\|_{L^2_\mu(\partial_+SM)}=\|I_a(D(h))\|_{L^2_\mu(\partial_+SM)} \le c\|h\|_\infty
\end{equation*}
so that $$- \log \Pi (f:\|I_af\|_{L^2_\mu(\partial_+SM)} \le \delta) \le - \log \tilde \Pi (h:\|h\|_{\infty} \le \delta/c)  $$ as well as  $$\|I_a(w-f_0)\|_{L^2_\mu(\partial_+SM)} = \|I_a(D(v-D^{-1}f_0))\|_{L^2_\mu(\partial_+SM)} \le c \|v-D^{-1}f_0\|_\infty$$ where $v = \sqrt {d_M} w \in V_{\tilde \Pi}$ corresponds to $w=Dv=d_M^{-1/2}v \in V_\Pi$. Thus by (\ref{cfcond}) the right hand side of (\ref{lowbd}) is bounded above by 
$$e^{2 (\delta_\varepsilon/\varepsilon)^2} e^{\phi_{\tilde \Pi, f_0}(\delta_\varepsilon/2c)} \Pi(D_\varepsilon^c) \le e^{(3 - D_0)(\delta_\varepsilon/\varepsilon)^2} \to 0$$ for $D_0>3$, completing the proof.
\end{proof}
For $\psi \in C^\infty(M)$ define now $\tilde \psi = -(I_a^*I_a)^{-1}\psi$. We have from Theorem \ref{main0} that $\tilde \psi$ can be written as $\tilde \psi = d_M^{-1/2} \bar \psi$ for some $\bar \psi \in C^\infty (M)$. Therefore, since the RKHS $V_{\tilde \Pi}$ of $\tilde \Pi$ contains $C^\infty(M)$ we have
\begin{equation}\label{rkhs}
\|\tilde \psi\|_{V_\Pi}^2 = \|\sqrt {d_M} \tilde \psi \|_{V_{\tilde \Pi}}^2 = \|\bar \psi\|_{V_{\tilde \Pi}}^2 \le C.
\end{equation} 
Next, the random variable $\langle \tilde \psi, f \rangle_{V_\Pi}, f \sim \Pi,$ is $\mathcal N(0,  \|\tilde \psi\|_{V_\Pi}^2)$ and the standard Gaussian tail inequality guarantees for all $u, \delta \ge0$ that $$\Pi\left( f: \frac{|\langle \tilde \psi, f \rangle_{V_\Pi}|}{\|\tilde \psi\|_{V_\Pi}} > u\frac{\delta}{\varepsilon} \right) \le e^{-u^2 (\delta/\varepsilon)^2/2}$$ hence Lemma \ref{postex} applies to the set $$D_\varepsilon = \left\{f: \frac{|\langle \tilde \psi, f \rangle_{V_\Pi}|}{\|\tilde \psi\|_{V_\Pi}} \le K\frac{\delta_\varepsilon}{\varepsilon} \right\}$$ whenever $K > \sqrt 6$, and in deriving the asymptotic distribution of the posterior measure we can restrict to the posterior distribution $\Pi^{D_\varepsilon}(\cdot|Y)$ arising from the prior $\Pi^{D_\varepsilon}=\Pi(\cdot \cap D_\varepsilon)/\Pi(D_\varepsilon)$. 

\begin{Proposition}
Assume Condition \ref{setting}. For $\psi \in C^\infty(M)$, define the random variables
\begin{equation} \label{effp}
\hat \Psi = \langle f_0, \psi \rangle_{H_1} - \varepsilon \langle I_a (I_a^*I_a)^{-1}\psi, \mathbb W\rangle_{H_2}.
\end{equation}
Then for all $\tau \in \mathbb R$ and as $\varepsilon \to 0$ we have
\begin{equation}\label{limit}
E^{\Pi_{D_\varepsilon}} \left[e^{\frac{\tau}{\varepsilon} \left(\langle f, \psi \rangle_{H_1} - \hat \Psi \right)} |Y\right]  = e^{\frac{\tau^2}{2} \|I_a(I_a^*I_a)^{-1}\psi\|_{H_2}^2} \times (1+o_{P_{f_0}^Y}(1)).
\end{equation}
\end{Proposition}
\begin{proof}
The left hand side of (\ref{limit}) equals, for $f_\tau = f+\tau \varepsilon \tilde \psi$,
\begin{align} \label{keyid}
&E^{\Pi_{D_\varepsilon}} \left[e^{\frac{\tau}{\varepsilon} \langle f -f_0, \psi \rangle_{H_1} +\tau \langle I_a (I_a^*I_a)^{-1} \psi, \mathbb W \rangle_{H_2} } |Y\right] \notag \\
&=  e^{\tau \langle I_a(I_a^*I_a)^{-1} \psi, \mathbb W \rangle_{H_2} }\frac{\int_\mathcal F e^{\frac{\tau}{\varepsilon} \langle f - f_0, \psi \rangle_{H_1} +\ell(f)-\ell (f_\tau) + \ell(f_\tau)} d\Pi^{D_\varepsilon}(f)}{\int_\mathcal F e^{\ell(f)} d\Pi^{D_\varepsilon}(f)} \notag \\
&= e^{\frac{\tau^2}{2} \|I_a\tilde \psi\|_{H_2}^2} \frac{\int_{D_\varepsilon} e^{\ell(f_\tau)} d\Pi(f)}{\int_{D_\varepsilon} e^{\ell(f)} d\Pi(f)}
\end{align}
since by Lemma \ref{LAN}
\begin{align*}
\ell(f)-\ell (f_\tau) &=-\frac{1}{2\varepsilon^2} \left(\| I_a(f - f_0)\|_{H_2}^2 - \|I_af - I_af_0 + \tau \varepsilon I_a\tilde \psi\|_{H_2}^2\right) + \tau \langle I_a\tilde \psi, \mathbb W \rangle_{H_2} \\
&= -\tau \langle I_a(I_a^*I_a)^{-1}\psi, \mathbb W \rangle_{H_2}  + \frac{\tau^2}{2} \|I_a\tilde \psi\|_{H_2}^2 + \frac{\tau}{\varepsilon} \langle I_a(f-f_0), I_a\tilde \psi\rangle_{H_2}
\end{align*}
and since by Theorem \ref{main0}
\begin{equation}\label{sieben}
 \langle I_a(f-f_0), I_a (I_a^*I_a)^{-1} \psi \rangle_{H_2} = \langle f-f_0, \psi \rangle_{H_1}.
\end{equation}
By the Cameron-Martin theorem (\cite{GN16}, Theorem 2.6.13) the last ratio in (\ref{keyid}) equals, for $\Pi_\tau$ the shifted law of $f_\tau, f \sim \Pi$,
$$ \frac{\int_{D_{\varepsilon, \tau}} e^{\ell(g)} \frac{d\Pi_\tau}{d\Pi}(g) d\Pi(g)}{\int_{D_\varepsilon} e^{\ell(g)} d\Pi(g)}= \frac{\int_{D_{\varepsilon, \tau}} e^{\ell(g)} e^{ \tau \varepsilon \langle \tilde \psi, g \rangle_{V_\Pi} - (\tau \varepsilon)^2 \|\tilde \psi\|_{V_\Pi}^2/2}  d\Pi(g)}{\int_{D_\varepsilon} e^{\ell(g)} d\Pi(g)}$$ where $D_{\varepsilon, \tau} =\{g=f_\tau: f \in D_\varepsilon\}$.
In view of (\ref{rkhs}) we have $\varepsilon\|\tilde \psi\|_{V_\Pi} \to 0$ as $\varepsilon \to 0$, and by definition of $D_\varepsilon$ and the Cauchy-Schwarz inequality we have convergence to zero of $$\varepsilon \sup_{g \in D_{\varepsilon, \tau}}|\langle \tilde \psi, g \rangle_{V_\Pi}| = \varepsilon \sup_{f \in D_\varepsilon}|\langle \tilde \psi, f + \tau \varepsilon \tilde \psi \rangle_{V_\Pi}| \le K \varepsilon \frac{\delta_\varepsilon}{\varepsilon} \|\tilde \psi\|_{V_\Pi} + |\tau| \varepsilon^2 \|\tilde \psi\|^2_{V_\Pi}$$
since $\delta_\varepsilon \to 0$. Conclude  that the last ratio is, for every $\tau \in \mathbb R$, $$ (1+o(1)) \frac{\int_{D_{\varepsilon, \tau}} e^{\ell(g)}  d\Pi(g)}{\int_{D_\varepsilon} e^{\ell(g)} d\Pi(g)} = (1+o(1)) \frac{\Pi(D_{\varepsilon, \tau}|Y)}{\Pi(D_{\varepsilon}|Y)}$$ as $\varepsilon \to 0$, and the proof is completed by showing that both the numerator and the denominator of the last ratio converge to one in probability: The denominator $\Pi(D_\varepsilon|Y)$ converges to one in $P_{f_0}^Y$-probability by Lemma \ref{postex}. The same is true for the numerator by applying Lemma \ref{postex} once more, since the Gaussian tail inequality  guarantees for $\sqrt 6<k<K$ and every $\tau \in \mathbb R$ that for $\varepsilon$ small enough that
\begin{align*}
\Pi(D^c_{\varepsilon, \tau}) &= \Pi\left(v: \frac{|\langle \tilde \psi, v- \tau \varepsilon \tilde \psi \rangle_{V_\Pi}|}{\|\tilde \psi\|_{V_\Pi}} > K\delta_\varepsilon/\varepsilon \right) \\
& \le \Pi\left(v: \frac{|\langle \tilde \psi, v \rangle_{V_\Pi}|}{\|\tilde \psi\|_{V_\Pi}} > K\frac{\delta_\varepsilon}{\varepsilon} - |\tau| \|\tilde \psi\|_{V_\Pi} \varepsilon \right)
\le e^{-k^2 (\delta_\varepsilon/\varepsilon)^2/2} .
\end{align*}
\end{proof}

Theorem \ref{main} now follows from the fact that convergence in total variation distance implies convergence in any metric for weak convergence, so that in view of Lemma \ref{postex} it suffices to prove the theorem with $\Pi^{D_\varepsilon}(\cdot|Y)$ replacing $\Pi(\cdot|Y)$, and using the previous proposition plus the fact that pointwise convergence of Laplace transforms (in probability) implies weak convergence (in probability), see, e.g., Proposition 30 in \cite{N17}.

\subsection{Proof of Theorem \ref{tikh}}

Let $(\Omega, \mathcal S, \Pr)$ be the probability space supporting the random variable $Y$ from (\ref{model0}) with law $P_{f_0}^Y$ (this space is implicitly constructed before (\ref{posto}) via the results from Section 7.3 in \cite{N17}). We show that  $\varepsilon^{-1} E^\Pi[\langle f, \psi \rangle_{H_1} - \hat \Psi |Y]$ converges to $0$ in $P_{f_0}^Y$-probability which implies the result since then by definition of $\hat \Psi$ we then have $$\varepsilon^{-1}\langle \bar f(Y) - f_0, \psi \rangle_{H_1} = -\langle I_a (I_a^*I_a)^{-1} \psi, \mathbb W \rangle_{H_2} + o_{P_{f_0}^Y}(1).$$  We argue by contradiction: Let $\varepsilon_m$ be any sequence such that $\varepsilon_m \to 0$ but assume $\varepsilon_m^{-1} E^\Pi[\langle f, \psi \rangle_{H_1} - \hat \Psi |Y]$ does not converge to $0$ in probability. Then there exists an event $\Omega' \in \mathcal S$ of positive probability $\Pr(\Omega')>0$ and $\xi>0$ such that along a subsequence of $m$,
\begin{equation}\label{tocont}
|\varepsilon_m^{-1} E^\Pi[\langle f, \psi \rangle_{H_1} - \hat \Psi |Y(\omega)]| \ge \xi~~\forall \omega \in \Omega'.
\end{equation} 
Since convergence in probability implies convergence almost surely along a subsequence, we can extract a further subsequence, still denoted by $\varepsilon_m$, for which we deduce from Theorem \ref{main} that $$\beta\big(\mathcal L( \varepsilon_m^{-1} (\langle f, \psi \rangle_{H_1} - \hat \Psi ) |Y), \mathcal L(Z) \big) \to 0$$ almost surely for $\Pr$, and where $Z$ is a $\mathcal N(0, \|I_a(I_a^*I_a)^{-1}\psi\|^2_{L^2_\mu(\partial_+(SM))})$ random variable.  Fix the event $\Omega_0 \subset \Omega$ of probability one where the last limit holds: then for every fixed $\omega \in \Omega_0$ we have the convergence in distribution $$\Psi_m (\omega) \equiv \varepsilon_m^{-1} (\langle f, \psi \rangle_{H_1} - \hat \Psi ) |Y(\omega) \to^d Z.$$ By Skorohod's theorem on almost surely convergent realisations of weakly convergent random variables (Theorem 11.7.2 in \cite{D02}) we can find, for every fixed $\omega \in \Omega_0$, a probability space on which we can define random variables $\tilde \Psi_m(\omega), \tilde Z$ such that $\mathcal L(\Psi_m(\omega))=\mathcal L(\tilde \Psi_m(\omega)), \mathcal L(Z)=\mathcal L(\tilde Z)$ and $$\tilde \Psi_m(\omega) - \tilde Z \to 0$$ almost surely as $m \to \infty$. By standard conjugacy arguments the law of $h=I_af|Y$ is a Gaussian measure on $L_\mu^2(\partial_+ SM)$. By Theorem \ref{main0}, when integrating against $\psi \in C^\infty(M)$ we see $$\langle f, \psi \rangle_{L^2(M)} = \langle h, I_a (I_a^*I_a)^{-1}  \psi \rangle_{L_\mu^2(\partial_+ SM)} $$ which is a well-defined normal distribution on the real line since the mapping $h \mapsto \langle h, I_a (I_a^*I_a)^{-1}  \psi \rangle_{L_\mu^2(\partial_+ SM)}$ from $L_\mu^2(\partial_+ SM) \to \mathbb R$  is linear and continuous in view of $I_a (I_a^*I_a)^{-1}  \psi \in L_\mu^2(\partial_+ SM)$, using Theorem \ref{main0} once more. Thus for every $\omega \in \Omega_0$ the variables $\tilde \Psi_m(\omega) -  \tilde Z,~m \in \mathbb N,$ are all  Gaussian and by the usual Paley-Zygmund argument (e.g., Exercise 2.1.4 in \cite{GN16}) almost sure convergence implies convergence of all moments, in particular $E|\tilde \Psi_m(\omega) - \tilde Z| \to 0$ as $m \to \infty$. From this we deduce, for all $\omega \in \Omega_0, \Pr(\Omega_0)=1,$ that $$ \varepsilon_m^{-1} E^\Pi[\langle f, \psi \rangle_{H_1} - \hat \Psi  |Y(\omega)]=E\Psi_m(\omega) = E[\tilde \Psi_m(\omega)] \to E \tilde Z = EZ=0$$ as $m \to \infty$, a contradiction to (\ref{tocont}) with $\Pr(\Omega')>0$, completing the proof.

\section{Proofs for Section \ref{xray}}\label{pxray}

In this section we prove Theorem \ref{main0} and we will do so by putting the theory into the framework of the transmission condition as developed in \cite{Ho0,Grubb2}.  We will give full details for the case of the geodesic X-ray transform $I$ and  indicate the (minor) modifications necessary for the proof to work also for the attenuated X-ray transform $I_{a}$ at the end. We note that previously known results only give that $I^*I$ is injective on $L^{2}(M)$ and surjectivity properties were only obtained after enlarging $M$ (as in \cite{PU}). These results are not sufficient to obtain the theorems in Section \ref{bvmsec}, nor do they expose the precise boundary behaviour as we do here.

\subsection{Setting up the scene and main ideas}\label{sec:roadmap}

We shall denote by $N$ the normal, or `information' operator $I^*I:L^{2}(M)\to L^{2}(M)$ introduced in Section \ref{xray}. An integral formula for $N$ can be derived directly from the expressions for $I$ and $I^*$:
\begin{equation}
Nf(x)=2\int_{S_{x}M}dv\int_{0}^{\tau(x,v)}f(\gamma_{x,v}(t))\,dt.
\label{eq:N}
\end{equation}

A property of fundamental importance is that whenever $(M,g)$ has no conjugate points, then, in the interior of $M$, the operator $N$ is an elliptic pseudo-differential operator ($\Psi$DO) of order $-1$ with principal symbol $c_d|\xi|^{-1}$, cf. \cite[Section 6.3]{GS}, \cite{SU} or Lemma 3.1 in \cite{PU}. [The reference \cite{GS} states this property under the so called {\it Bolker condition}, which is seen to be equivalent in our case to the absence of conjugate points.] We refer to \cite{Treves} for a treatment of $\Psi$DOs. In particular recall that for $P$ a classical $\Psi$DO of order $m\in \mathbb{C}$, a full symbol in local coordinates is denoted by $p(x,\xi)\sim \sum_{j=0}^{\infty}p_{j}(x,\xi)$ where $p_{j}(x,t\xi)=t^{m-j}p_{j}(x,\xi)$, and where $p_0$ is the {\it principal symbol}. The operator $P$ is {\it elliptic} if $p_0(x,\xi)\ne 0$ for all $(x,\xi)$ in the cotangent bundle, $\xi\ne 0$. 

Recall that $(M,g)$ is called {\it simple} if it is non-trapping, has strictly convex boundary and no conjugate points. Simple manifolds are simply connected; in fact they are diffeomorphic to balls in Euclidean space. From now on we shall assume that $(M,g)$ is simple.
It will be convenient for what follows to consider $(M, g)$ isometrically embedded into a closed manifold $(S,g)$. Since $M$ is simple, there is an open neighborhood $U_{1}$ of $M$ in $S$, such that its closure $M_{1}:=\overline{U}_{1}$  is a compact simple manifold, see Fig. \ref{fig:M}. Let $I_{1}$ denote the geodesic ray transform associated to $(M_{1},g)$ and let $N_{1}=I_{1}^*I_{1}$.

\begin{figure}[htpb]
    \centering
    \includegraphics[height=0.25\textheight]{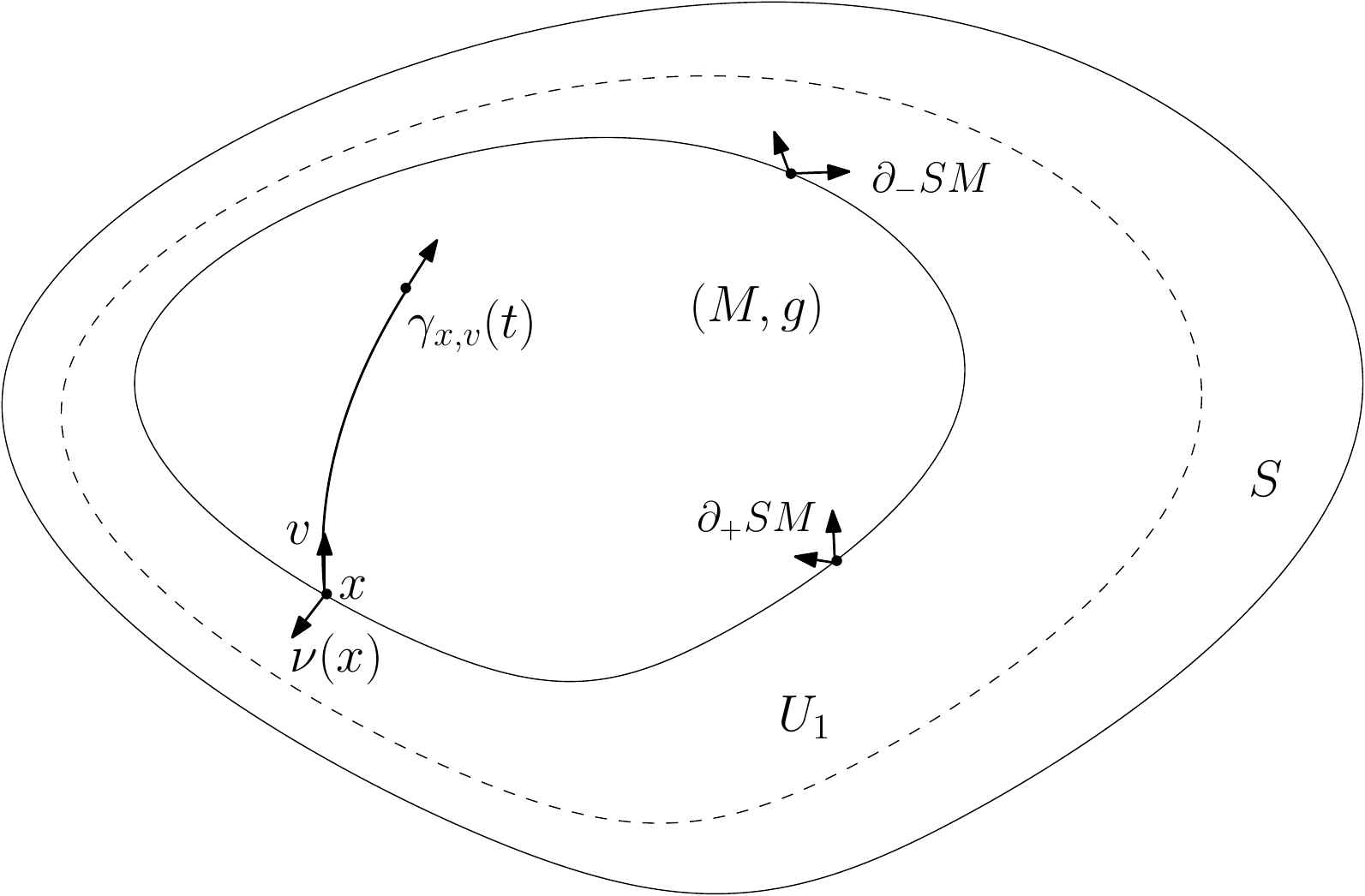}
    \caption{Setting}
    \label{fig:M}
\end{figure}

%Let $\psi$ be a smooth function on $S$ with support contained in $U_{1}$ and such that it is equal to $1$ near $M$. Let $\Delta_{g}$ denote the Laplacian of $(S,g)$. Define 
%\begin{align}
   % P:=\psi N_{1} \psi+(1-\psi)(1+\Delta_{g})^{-1/2}(1-\psi).    
    %\label{eq:P}
%\end{align}
%As we have already mentioned, $N_{1}$ is an elliptic $\Psi$DO of order $-1$ on $U_{1}$ and thus
%$P$ is also an elliptic $\Psi$DO of order $-1$ in $S$. 

Following \cite{PU} we may cover $(S,g)$ with finitely many simple open sets $U_{k}$ with $M\subset U_{1}$, $M\cap \overline{U}_{j}=\emptyset$ for $j\geq 2$, and consider a partition of unity $\{\varphi_{k}\}$ subordinate to $\{U_{k}\}$ so that $\varphi_{k}\geq 0$, $\text{supp}\,\varphi_{k}\subset U_{k}$ and $\sum\varphi_{k}^2=1$. We pick $\varphi_{1}$ such that $\varphi_{1} \equiv 1$ on a neighborhood of $M$ compactly supported in $U_1$. Hence, for $I_k$ the ray transform associated to $(U_k, g)$, we can define
\begin{equation}    
    Pf:=\sum_{k}\varphi_{k}(I^{*}_{k}I_{k})(\varphi_{k} f), \qquad f\in C^{\infty}(S).
    \label{eq:P}
\end{equation}
Each operator $I^{*}_{k}I_{k}:C^{\infty}_{c}(U_{k})\to C^{\infty}(U_{k})$ is an elliptic $\Psi$DO of order $-1$ and principal symbol $c_{d}|\xi|^{-1}$, and hence so is $P$.  Having $P$ defined on a closed manifold is convenient, since one can use standard mapping properties for $\Psi$DOs.  For instance for $P$ defined by \eqref{eq:P} we have
\[ P:H^{s}(S)\to H^{s+1}(S) \qquad \text{for all } s\in \mathbb{R}, \]
where $H^{s}(S)$ denotes the standard $L^2$ Sobolev space of the closed manifold $S$  (when $s$ is a nonnegative integer, $H^{s}(S)$ can be identified with the set of $u\in L^{2}(S)$ such that $Du\in L^{2}(S)$ for all differential operators $D$ of order $\leq s$ with coefficients in $C^{\infty}(S)$, see \cite{T} for the definition for arbitrary $s\in \mathbb{R}$).

Let $r_{M}:L^{2}(S)\to L^{2}(M)$ denote restriction to $M^{\text{int}}$, the interior of $M$, and $e_{M}:L^{2}(M)\to L^{2}(S)$ extension by zero. (We could consider restriction to $M$ as well, but this makes no difference since the boundary of $M$ has measure zero.)
Both operators are bounded and dual to each other. Since $\varphi_{1}=1$ near $M$, given $f\in C^{\infty}_{c}(M^{\text{int}})$ (smooth functions with compact support contained in $M^{\text{int}}$) we have 
\[r_{M}Pe_{M}f=r_{M}N_{1}\varphi_{1} e_{M}f=r_{M}N_{1}e_{M}f.\]
Equation (\ref{eq:N}) shows that $r_{M}N_{1}e_{M}f=Nf$ and thus by density of $C^{\infty}_{c}(M^{\text{int}})$ in $L^{2}(M)$, we have that $P$ and $N$ are related by the following truncation process:
\begin{equation}
    N = r_{M}Pe_{M}\qquad\text{in}\;L^{2}(M).
    \label{eq:truncation}
\end{equation}
Since $P:L^{2}(S)\to H^{1}(S)$, this gives immediately the mapping property $N:L^{2}(M)\to H^{1}(M)$ when the spaces $H^s(M)$ are defined by restriction (\ref{restnorm}). Since the embedding  $H^{1}(M)\hookrightarrow L^{2}(M)$ is compact, obviously $N:L^{2}(M)\to L^{2}(M)$ is compact (and hence $I$).

However, without further analysis not much more can be said about the mapping properties of $N$, especially if we are interested in functions supported all the way to the boundary of $M$. From (\ref{eq:truncation}), we see that $e_{M}$ could produce singularities for higher order Sobolev spaces, preventing good mapping properties in `smooth topologies'. A key input of Boutet de Monvel \cite{BdeM1,BdeM2} (see also \cite{CP}) was to show that a necessary and sufficient condition for $P$, a $\Psi DO$ of order $m$ defined on $S$, to satisfy $r_M P e_M (C^\infty(M))\subset C^\infty(M)$, is that $P$ satisfies the {\it transmission condition} with respect to $\partial M$ in the sense that 
\begin{align}
    \partial_x^\beta \partial_\xi^\alpha p_j(x,\nu(x)) = e^{\pi i (m-j-|\alpha|)} \partial_x^\beta \partial_\xi^\alpha p_j(x,-\nu(x)),
    \label{eq:transmission}
\end{align}
for all $j$, $\alpha$, $\beta$ and $x\in \partial M$.%, where $\nu$ denotes the outward unit normal to $\partial M$. 

Unfortunately such a condition does not hold in the case of $N_1$ (or $P$) defined above, as the following example shows: let $M$ be the unit disk in $\mathbb{R}^{2}$. An elementary calculation gives that $N(1)=4E(r)/\pi$ where $E$ is the complete elliptic integral of the second kind and $r$ is the radial coordinate. As $r$ approaches $1$, $E'(r)$ blows up and hence $N(1)\notin C^{\infty}(M)$, therefore $N = r_M N_1$ cannot satisfy the transmission condition. 

Furthermore for purposes of inversion, even if $h\in C^\infty(M)$, we cannot expect the solutions $f$ to $N(f) = h$ to be in $C^\infty(M)$ either: in the previous example, it is not hard to check that $I((1-r^{2})^{-1/2})$ ($r$ denoting distance to the origin) is a constant function \cite[Corollary 3.3]{L} and therefore so is 
\begin{equation}
    N((1-r^{2})^{-1/2})=c.
    \label{eq:ex}
\end{equation}
%Moreover, in this example it is not hard to check that $I((1-r^{2})^{-1/2})$ is a constant function \cite[Corollary 3.3]{L} and therefore so is 
%\begin{equation}
%    N((1-r^{2})^{-1/2})=c.
%    \label{eq:ex}
%\end{equation}
%Thus we cannot expect the solutions $f$ to $N(f)=h$ to be in $C^{\infty}(M)$ even if $h\in C^{\infty}(M)$.  

While $N_1$ does not satisfies condition \eqref{eq:transmission}, we show that it satisfies a modified transmission condition as introduced by H\"ormander in \cite{Ho0} and recently expanded and enhanced by Grubb in \cite{Grubb2}. Namely, given $\mu\in {\mathbb C}$ with real part $\Re\mu >-1$, we say that $P$, a $\Psi DO$ of order $m$ defined on $S$, satisfies a {\it transmission condition of type $\mu$} with respect to $\partial M$ if its symbol satisfies
\begin{align}
    \partial_x^\beta \partial_\xi^\alpha p_j(x,\nu(x)) = e^{\pi i (m-2\mu-j-|\alpha|)} \partial_x^\beta \partial_\xi^\alpha p_j(x,-\nu(x)),
    \label{eq:mutransmission}
\end{align}
for all $j$, $\alpha$, $\beta$ and $x\in \partial M$, generalizing the case $\mu=0$ given by \eqref{eq:transmission}. To tie this condition with mapping properties, using $\mu$ as above, let us define
\[ \mathcal{E}_{\mu}(M) := \{ e_{M} d_{M}(x)^{\mu}\varphi, \quad \varphi\in C^{\infty}(M)\}, \]
where $d_M(x)$ is a $C^\infty(M)$-function equal to $dist(x,\partial M)$ near $\partial M$ and positive on the interior of $M$. Then, as a generalization to Boutet de Monvel's result above, the following theorem appears in \cite[Theorem 18.2.18]{Ho}:
\begin{Theorem}\label{thm:hor} A necessary and sufficient condition in order that $r_{M}Pu\in C^{\infty}(M)$ for all $u\in \mathcal{E}_{\mu}(M)$ is that $P$ satisfies the $\mu$-transmission condition \eqref{eq:mutransmission}.
\end{Theorem}
%\begin{Theorem} A necessary and sufficient condition in order that $r_{M}Ru\in C^{\infty}(M)$ for all $u\in \mathcal{E}_{\mu}(M)$ is that $R$ satisfies the $\mu$-transmission condition, namely that
%\[\partial_{x}^{\beta}\partial_{\xi}^{\alpha}p_{j}(x, \nu)=e^{\pi i(m-2\mu-j-|\alpha|)}\partial_{x}^{\beta}\partial_{\xi}^{\alpha}p_{j}(x,-\nu),\]
%for all $j$, $\alpha$ and $\beta$ and $x\in\partial M$, where $\nu$ denotes the outward unit normal to $\partial M$.
%\label{thm:hor}
%\end{Theorem}
To make use of the theorem above, we first prove in Section \ref{sec:transmission} that
\begin{Lemma}\label{lem:transmission} The operator $P$ defined in \eqref{eq:P} satisfies the transmission condition of type $\mu=-1/2$ with respect to $\partial M$. 
\end{Lemma}
In particular, Lemma \ref{lem:transmission} and Theorem \ref{thm:hor} imply that 
\[ N:d_{M}^{-1/2}C^{\infty}(M) \to C^\infty(M) \]
is well-defined. Notice that the domain allows for functions which blow up near the boundary like $dist(x,\partial M)^{-\frac{1}{2}}$, explaining \eqref{eq:ex}.

After constructing in Section \ref{section:S+H} appropriate Hilbert-scale versions of ${\mathcal E}_{-1/2}(M)$, namely, the H\"ormander spaces $H^{-1/2(s)}(M)$, the first basic result \cite[Theorem 4.2]{Grubb2} applied to $r_{M}P$ gives further mapping properties: 
\begin{Theorem} 
    $r_{M}P$ maps $H^{-1/2(s)}(M)$ continuously into $H^{s+1}(M)$, where $P$ is defined in \eqref{eq:P} and $s>-1$.
\end{Theorem}

While all the results above only discuss forward mapping properties, using ellipticity will show that such an operator is in fact Fredholm in the functional settings mentioned above. Then proving that its kernel and co-kernel are trivial will ensure that it will be invertible in these settings as well. In particular, the main result we prove below provides a full solution to the homogeneous Dirichlet problem for $P$ on the domain $M$.

\begin{Theorem}\label{thm:mainiso} Let $P$ be the elliptic $\Psi$DO of order $-1$ given by \eqref{eq:P}. For $s>-1$ the map $r_{M}P:H^{-1/2(s)}(M)\to H^{s+1}(M)$ is a homeomorphism.
Moreover, $N:d_{M}^{-1/2}C^{\infty}(M)\to C^{\infty}(M)$ is a bijection.
\end{Theorem}

%Depending on what is required, it is still possible to derive surjectivity properties and stability estimates avoiding a detailed study of the boundary behaviour. A typical example is the following result due to Pestov and Uhlmann \cite[Theorem 3.1]{PU} that has been crucial for the solution of several geometric inverse problems: for any $h\in H^{s}(M)$, $s\geq 0$, there exists $f\in H^{s-1}(S)$ such that $r_{M}N_{1}(\varphi_{1} f)=h$. Moreover, if $h\in C^{\infty}(M)$, then there exists $f\in C^{\infty}(S)$ such that $r_{M}N_{1}(\varphi_{1} f)=h$. This theorem can be proved with the setting described up to now combined with the injectivity of $I$ on $C^{\infty}(M)$ \cite{Mu}. Theorem \ref{thm:mainiso} above is a version of this theorem which incorporates the behaviour at the boundary and the features observed in (\ref{eq:ex}).

The outline of the remainder is as follows. Section \ref{section:S+H} contains details on Sobolev and H\"ormander spaces. Section \ref{sec:transmission} will be devoted to the proof of Lemma \ref{lem:transmission}, and Section \ref{sec:thmthm}  to the proof of Theorem \ref{thm:mainiso}, requiring a few technical lemmas, followed by the proof of Theorem \ref{main0}.

%Below we shall obtain a version of the theorem which incorporates the behaviour at the boundary and the features observed in (\ref{eq:ex}).

%\F{Finally, we explain in Section \ref{sec:consequence} how Theorem \ref{thm:mainiso} allows us to revisit stability estimates existing in prior contexts.}

%From (\ref{eq:truncation}), we see that $e_{M}$ could produce singularities for higher order Sobolev spaces and this can be observed in the simplest of cases, like $M$ being the unit disk in $\mathbb{R}^{2}$. An elementary calculation gives that $N(1)=4E(r)/\pi$ where $E$ is the complete elliptic integral of the second kind and $r$ is the radial coordinate. As $r$ approaches $1$, $E'(r)$ blows up and hence $N(1)\notin C^{\infty}(M)$.
%Moreover, in this example it is not hard to check that $I((1-r^{2})^{-1/2})$ is a constant function \cite[Corollary 3.3]{L} and therefore so is
%\begin{equation}
%N((1-r^{2})^{-1/2})=c.
%\label{eq:ex}
%\end{equation}
%Thus we cannot expect the solutions $f$ to $N(f)=h$ to be in $C^{\infty}(M)$ even if $h\in C^{\infty}(M)$.  The correct solution space will be fully described below. 
%%We note that for this example it is natural to take $\mathbb{R}^2$ as ambient space and $P=(-\Delta)^{-1/2}$.

%As a consequence we see that $N_{1}$ {\bf does not} satisfy the {\it transmission} condition of Boutet de Monvel (see \ \cite{BdeM1, BdeM2, CP}). 

%However as we shall see, $N_{1}$ satisfies a modified transmission condition as introduced by H\"ormander in \cite{Ho0} and recently expanded and enhanced in \cite{Grubb2}.

\subsection{Sobolev spaces and H\"ormander spaces} \label{section:S+H} 

In this section we summarize the main functional setting that we will be using.
Here we shall be concerned only with $L^{2}$-Sobolev and H\"ormander spaces. The Sobolev spaces are standard but the H\"ormander spaces are less so. For the latter we will follow \cite{Grubb2} and for the former \cite{McL,T} (with minor departures in notation). As before we let $(M,g)$ be a compact Riemannian manifold with boundary which we
think isometrically embedded into a closed manifold $(S,g)$.
We write $H^{s}(S)$ for the standard $L^2$ based Sobolev space of the closed manifold $S$. We denote
\[H^{s}(M)=r_{M}H^{s}(S)=\{u|_{M^{\text{int}}}:\;u\in H^{s}(S)\}\]
%\marginpar{Most common would be $H^{s}(M^{\text{int}})$? Notations differ widely}
equipped with the quotient norm 
\begin{equation} \label{restnorm}
\norm{u}_{H^{s}(M)}:=\inf\{\norm{w}_{H^{s}(S)}:\;w\in H^{s}(S),\;r_{M}w=u\}.
\end{equation}
We denote
\[H^{s}_{M}(S):=\{u\in H^{s}(S):\;\text{supp}(u)\subset M\}.\]
\begin{Remark}{\rm The space $H_{M}^{s}(M)$ can also be seen as the closure of $C^{\infty}_{c}(M^{\text{int}})$ in $H^{s}(S)$. Finally we can also define
$H_{0}^{s}(M)$ as the closure of $C^{\infty}_{c}(M^{\text{int}})$ in $H^{s}(M)$. When $s\notin \mathbb{Z}+\frac{1}{2}$, there is a natural identification with $H^{s}_{M}(S)$. When $s$ is a nonnegative integer, $H^{s}(M)$ can be identified with the set of $u\in L^{2}(M)$ such that $Du\in L^{2}(M)$ for all differential operators $D$ of order $\leq s$ with coefficients in $C^{\infty}(M)$.}
\end{Remark}

One of the main inputs of \cite{Grubb2} is the introduction of particularly efficient {\it order reducing operators}, cf. \cite[Theorem 1.3]{Grubb2}. These are classical elliptic $\Psi$DOs on $S$ of order $\mu$ (denoted $\Lambda_{+}^{(\mu)}$) preserving support in $M$ and defining homeomorphisms
\begin{equation}
\Lambda_{+}^{(\mu)}:H^{s}_{M}(S)\to H^{s-\Re\mu}_{M}(S),
\label{eq:homeoL}
\end{equation}
where $\Re\mu$ denotes the real part of $\mu$.
These operators are used to define the H\"ormander spaces (also known as $\mu$-transmission spaces)
\[H^{\mu(s)}(M):=\Lambda_{+}^{(-\mu)}e_{M}H^{s-\Re\mu}(M),\;\;\;s>\Re\mu-1/2.\]
For $s>\Re\mu-1/2$, the maps $r_{M}\Lambda_{+}^{(\mu)}:H^{\mu(s)}(M)\to H^{s-\Re\mu}(M)$ are homeomorphisms with inverse
$\Lambda_{+}^{(-\mu)}e_{M}$ \cite[Proposition 1.7]{Grubb2}. We have a natural embedding $H_{M}^{s}(S)\subset H^{\mu(s)}(M)$. 

As explained in Section \ref{sec:roadmap}, these spaces are specifically adapted to the $\mu$-transmission condition \eqref{eq:mutransmission} and will provide natural spaces of solutions to the equation $Nf=h$ where $\mu=-1/2$.

The order reducing operators $\Lambda_{+}^{(\mu)}$ are used in conjunction with their adjoints $\Lambda_{-}^{(\bar{\mu})}$ by considering a new operator $Q=\Lambda_{-}^{(\mu-m)}P\Lambda_{+}^{(-\mu)}$, where $m$ is the order of $P$. The point is that if $P$ satisfies the $\mu$-transmission condition \eqref{eq:mutransmission}, then $Q$ satisfies the transmission condition \eqref{eq:transmission} with $\mu=0$ and fits the Boutet de Monvel calculus. This is the main idea in \cite{Grubb2}.

Let $d_{M}(x)$ be a $C^{\infty}(M)$-function equal to $dist(x,\partial M)$ near $\partial M$ and positive on the interior of $M$. For $\mu\in\mathbb {C}$ with $\Re \mu>-1$, let $\mathcal{E}_{\mu}(M)$ denote the space of functions $u$ such that $u=e_{M}d_{M}(x)^{\mu}\varphi$ with $\varphi\in C^{\infty}(M)$. One can show that \cite[Proposition 4.1]{Grubb2}:
\[\mathcal{E}_{\mu}(M)=\bigcap_{s}H^{\mu(s)}(M).\]

The spaces $H^{\mu(s)}(M)$ were introduced in \cite{Ho0} as the completion of $\mathcal{E}_{\mu}(M)$ in the topology defined by the seminorms $u\mapsto \norm{r_{M}Pu}_{H^{s-\Re m}(M)}$, where $P$ runs through the operators satisfying the $\mu$-transmission condition (see below) and any order $m\in\mathbb{C}$.  H\"ormander's starting point was the work of Vishik and Eskin \cite{VE,E}.

\subsection{Proof of Lemma \ref{lem:transmission}} \label{sec:transmission}

%\F{[F: previous title removed (The $\mu$-transmission condition) removed]}

%Let $R$ be any classical $\Psi$DO of order $m\in \mathbb{C}$ in $S$ with symbol in local coordinates $p(x,\xi)\sim \sum_{j=0}^{\infty}p_{j}(x,\xi)$ where $p_{j}(x,t\xi)=t^{m-j}p_{j}(x,\xi)$.

\begin{Definition} {\rm We shall say that $P$ has {\it even} symbol if $p_{j}(x,-\xi)=(-1)^j p_{j}(x,\xi)$ for all $j\geq 0$. It is easy to check that this condition is independent of the coordinates chosen. (Recall that the full symbol is not defined intrinsically.)}    
\end{Definition}

\begin{Lemma}\label{lem:even} The symbol of $N$ (or $N_1$) is even.
\end{Lemma}

\begin{proof} There are (at least) three possible proofs of this lemma.  As explained in \cite{SSU} the full {\it geometric} symbol of $N$ coincides with its principal symbol $c_{d}|\xi|^{-1}$. In \cite{Sh} a relation is established between the full geometric symbol and the ordinary full symbol in local coordinates. See for example equations (1.6) and (1.7) in \cite{Sh}.
An inspection of those formulas shows that the symbol is even starting from the fact that $c_{d}|\xi|^{-1}$ is even.
Another more natural proof was suggested to us by Gunther Uhlmann and is based on the calculation of the full symbol in \cite{SU}.
Equation (17) in \cite{SU} gives an explicit formula for the amplitude $M(x,y,\xi)$ of the $\Psi$DO for the case of 2-tensors. In the case of functions the formula is 
\[M(x,y,\xi)=\int e^{-i\xi\cdot z}(G^{(1)}z\cdot z)^{\frac{-n+1}{2}}\frac{|\det G^{(3)}|}{\sqrt{\det g}}\,dz,\]
where $G^{1}(x,y)$ and $G^{3}(x,y)$ are defined in \cite[Lemma 3]{SU}, but we do not need to know what they are. The terms $p_{j}(x,\xi)$ may be derived from the amplitude by
\[p_{j}(x,\xi)=\sum_{|\alpha|=j}\frac{1}{\alpha !}\partial_{\xi}^{\alpha}D^{\alpha}_{y} M(x,y,\xi)|_{y=x}.\]
Since $M(x,y,\xi)$ is even in $\xi$, we see that once we start taking derivatives in $\xi$,  the parity of $p_{j}$ in $\xi$ changes according to $(-1)^{j}=(-1)^{|\alpha|}$.

The quickest way is perhaps to use \cite[Lemma B.1]{DPSU} which covers a broad range of operators for the
form 
\[Af(x)=\int_{S_{x}U_{1}}\int_{\mathbb R} A(x,r,w)f(x+rw)\,dr\,dS_{x}(w).\]
Our operator $N$ is certainly of this form (after some change of variables).  The lemma proves that $A$ is a classical $\Psi$DO of order $-1$ and computes explicity the full symbol deriving a formula
\[p_{k}(x,\xi)=2\pi\frac{i^k}{k!}\int_{S_{x}U_{1}}\partial^{k}_{r}A(x,0,w)\delta^{(k)}(w\cdot\xi)\,dS_{x}(w).\]
From this formula we see right away that $p_{k}(x,-\xi)=(-1)^{k}p(x,\xi)$ since the delta function $\delta$ is even.
\end{proof}

\begin{proof}[Proof of Lemma \ref{lem:transmission}] In a tubular neighbourhood of $\partial M$, the full symbol of $P$ coincides with that of $N_1$. The result is then a direct consequence of Lemma \ref{lem:even} and the fact that $m=-1$.
\end{proof}

\subsection{Proof of Theorems \ref{thm:mainiso} and \ref{main0}}\label{sec:thmthm}

To prove Theorem \ref{thm:mainiso}, a first step is to prove that for $P$ defined in \eqref{eq:P}, $r_M P$ is a Fredholm operator in the functional settings $H^{-1/2(s)}(M)\to H^{s+1}(M)$ for $s>-1$, and $\mathcal{E}_{-1/2}(M)\to C^{\infty}(M)$. This is mainly due to the ellipticity of $P$, and one additional concept from \cite{Grubb2}, {\it the factorization index $\mu_0$}. This is defined for elliptic operators of order $m$ as
\[\mu_{0}:=m/2+(a_{+}-a_{-})/2\pi i\]
where $\exp a_{\pm}=p_{0}(x,\mp\nu(x))$. For the case of $N_{1}$, we have $\mu_0=-1/2$ since the principal symbol of $N_{1}$ is (up to a constant)
$|\xi|^{-1}$, cf. \cite[Example 3.2]{Grubb2}. 
In general, for elliptic operators of order $m$ with even symbol, $\mu=\mu_{0}=m/2$. Then as a direct consequence of \cite[Theorem 4.4]{Grubb2}, we obtain: 

\begin{Theorem}\label{theorem:fred}
    Assume $s>-1$.  Suppose $u\in H_{M}^{\sigma}(S)$ for some $\sigma>-1$ and let $P$ be the elliptic $\Psi$DO of order $-1$ given by \eqref{eq:P}. If $r_{M}Pu\in H^{s+1}(M)$, then $u\in H^{-1/2(s)}(M)$. Moreover, the mapping $r_{M}P: H^{-1/2(s)}(M)\to H^{s+1}(M)$ is Fredholm. In particular, if $r_{M}Pu\in C^{\infty}(M)$, then $u\in \mathcal{E}_{-1/2}(M)$. The mapping $r_{M}P:  \mathcal{E}_{-1/2}(M)\to C^{\infty}(M)$ is also Fredholm. 
\end{Theorem}

%We now state our main result which provides a full solution to the homogeneous Dirichlet problem for $P$ on the domain $M$.

%\begin{Theorem}\label{thm:mainiso} Let $P$ be the elliptic $\Psi$DO of order $-1$ given by \eqref{eq:P}. For $s>-1$ the map $r_{M}P:H^{-1/2(s)}(M)\to H^{s+1}(M)$ is a homeomorphism.
%Moreover, $N:d_{M}^{-1/2}C^{\infty}(M)\to C^{\infty}(M)$ is a bijection.
%\end{Theorem}

Before proving Theorem \ref{thm:mainiso}, we state and prove a few preliminary lemmas. We begin with:

\begin{Lemma}\label{lem:keyfact}
    If $(M,g)$ is a non-trapping manifold with strictly convex boundary, there exists a constant $C_0>0$ such that 
    \begin{align*}
	d_{M}(\gamma_{x,v}(t),\partial M) \ge C_0\ t(\tau(x,v)-t),\;\text{for all}\; (x,v)\in \partial_+ SM, \;\; t\in [0,\tau(x,v)]. 
    \end{align*}    
\end{Lemma}

\begin{proof} In what follows, denote $\alpha_A:\partial_+ SM\to \partial_+ SM$ the {\em antipodal scattering relation}, that is, the map $(x,v)\mapsto (\gamma_{x,v}(\tau(x,v)), -\dot\gamma_{x,v}(\tau(x,v)))$. $\alpha_A$ satisfies $\alpha_A^2 = Id$ and if $g(\nu_x,v)=0$, then $\alpha_A (x,v) = (x,-v)$.

 It is enough to show that the nonnegative function $F(x,v,t) := \frac{d_{M}(\gamma_{x,v}(t),\partial M)}{t(\tau(x,v)-t)}$ is uniformly bounded away from zero on the set 
    \begin{align*}
	G = \{ (x,v)\in \partial_+ SM, \quad t\in (0,\tau(x,v))\}.
    \end{align*}    
    In what follows, we will also use \cite[Lemma 4.1.2 p113]{Sharafudtinov1994} stating that there exists $C_2>0$ such that 
    \begin{align}
	\tau(x,v) \le C_2 |\langle \nu(x),v\rangle_{g}|, \qquad (x,v)\in \partial_+ SM, \quad \langle\nu(x),v\rangle_{g} \neq 0,
	\label{eq:C2}
    \end{align}
    where $\nu(x)$ is the outer unit normal at $x\in \partial M$. This is essentially a consequence of the strict convexity of the boundary of $M$.
    Using that $d_{M}(\gamma_{x,v}(t),\partial M)|_{t=0} = 0$ and $\frac{d}{dt} d_{M}(\gamma_{x,v}(t),\partial M)|_{t=0} = |\langle \nu(x),v\rangle_{g} |$ (e.g., by using normal geodesic coordinates), l'H\^opital's rule implies
    \begin{align*}
	\lim_{t\to 0^+} F(x,v,t) = \frac{|\langle \nu(x),v\rangle_{g}|}{\tau(x,v)} = \frac{\mu(x,v)}{\tau(x,v)} \ge \frac{1}{C_2}.
    \end{align*}
    Moreover, since $\gamma_{x,v}(t) = \gamma_{\alpha_A(x,v)}(\tau(x,v)-t)$, we have the symmetry property $F(x,v,t) = F(\alpha_A(x,v), \tau(x,v)-t)$, and this allows to deduce the limit
    \begin{align*}
	\lim_{t\to \tau(x,v)^-} F(x,v,t) = \frac{\mu (\alpha_A(x,v))}{\tau(\alpha_A(x,v))} \ge \frac{1}{C_2}.
    \end{align*}
    By compactness, the result follows since $F$ is uniformly bounded away from zero outside any neighborhood of $\{t=0\}\cup \{t=\tau(x,v)\}$ in $G$.     
\end{proof}

An important ingredient in what follows is the consideration of the following weighted space $L^{2}(M,d_{M}^{1/2})$ where the measure is $d_{M}^{1/2}dx$. 
Recall that $I^*$ denotes the usual backprojection, i.e. the adjoint of $I:L^2(M)\to L^2_\mu (\partial_+ SM)$.

\begin{Lemma} \label{lemma:weight} The following hold: 
  \begin{itemize}
    \item[$(i)$] The map $I:L^{2}(M,d_{M}^{1/2})\to L_{\mu}^{2}(\partial_{+} SM)$ is bounded with adjoint $I^*_{w}=d_{M}^{-1/2}I^{*}$. 
    \item[$(ii)$] The map $I^*: L^2_\mu(\partial_+ SM) \to L^{2}(M,d_{M}^{-1/2})$ is bounded. 
  \end{itemize}
\end{Lemma}

\begin{proof} Let $f\in L^2(M, d_{M}^{1/2})$ and write $f = d_{M}^{-1/4} g$ for some $g\in L^2(M)$. We write
    \begin{align*}
	I f(x,v) &= \int_0^{\tau(x,v)} g(\gamma_{x,v}(t)) \frac{dt}{d_{M}(\gamma_{x,v}(t),\partial M)^{1/4}}, \\
	|I f(x,v)|^2 &\le \int_0^{\tau(x,v)} |g(\gamma_{x,v}(t))|^2\ dt\ \int_{0}^{\tau(x,v)} \frac{dt}{d_{M}(\gamma_{x,v}(t),\partial M)^{1/2}},
    \end{align*}
    using Cauchy-Schwarz inequality. Using Lemma \ref{lem:keyfact}, we have, for any $(x,v)\in \partial_+ SM$,
    \begin{align*}
	\int_{0}^{\tau(x,v)} \!\!\!\! \frac{dt}{d_{M}(\gamma_{x,v}(t),\partial M)^{1/2}} \le \frac{1}{C_{0}^{1/2}} &\int_0^{\tau(x,v)} \!\!\!\! \frac{dt}{(t(\tau-t))^{\frac{1}{2}}} \stackrel{t=\tau u}{=} \frac{1}{C_0^{1/2}} \int_0^1 \frac{du}{(u(1-u))^{\frac{1}{2}}} \\= &\frac{\pi}{C_0^{1/2}}. 
    \end{align*}
    Integrating over $\partial_+ SM$, we then obtain
    \begin{align*}
	\int_{\partial_+ SM} |I f(x,v)|^2\ d\mu &\le \frac{\pi}{C_0^{1/2}} \int_{\partial_+ SM} \int_0^{\tau(x,v)} |g(\gamma_{x,v}(t))|^2\ dt\ d\mu \\
	&= \frac{\pi}{C_0^{1/2}} \int_{SM} |g(x)|^2\,dxdv \qquad (\text{by Santal\'o's formula \cite{Sharafudtinov1994}}) \\
	&= \frac{\pi \Vol(S^{d-1})}{C_0^{1/2}} \int_M |g(x)|^2 \ dx \\
	&= \frac{\pi\Vol(S^{d-1})}{C_0^{1/2}} \|f\|_{L^2(M,d_{M}^{1/2})}^2,
    \end{align*}
    hence $(i)$ holds. Then $(ii)$ is a direct consequence of the factorization
    \begin{align*}
      I^*: L^2_\mu (\partial_+ SM) \stackrel{I_w^*}{\longrightarrow} L^2(M, d_{M}^{1/2}) \stackrel{d_{M}^{1/2}\cdot}{\longrightarrow} L^2(M, d_{M}^{-1/2}),
    \end{align*}
    where $I_w^*$ is continuous by $(i)$ and the second operator (multiplication by $d_{M}^{1/2}$) is an isometry in the setting above. The proof of Lemma \ref{lemma:weight} is complete.  
\end{proof}

\begin{Lemma}\label{lem:phiL1}
    Given $\varphi\in C^{\infty}(M)$, $d_{M}^{-1/2}\varphi\in L^{2}(M,d_{M}^{1/2})\cap L^{1}(M)$.
\end{Lemma}

\begin{proof} Obviously $\varphi$ is bounded in $M$. Since $(d_{M}^{-1/2}\varphi)^2 d_{M}^{1/2}=d_{M}^{-1/2}\varphi^{2}$, we just need to prove that $d_{M}^{-1/2}$ is in $L^1(M)$. By taking local geodesic normal coordinates where $x_{n}$ denotes distance to the boundary, the lemma follows from the elementary observation
    \begin{align*}
	\int_{0}^{\varepsilon}x_{n}^{-1/2}\,dx_{n}<\infty
    \end{align*}
    since locally $d_{M}=x_{n}$.
\end{proof}

\begin{Remark}{\rm The same proof shows that $d_{M}^{-1/2}H^{s}(M)\subset  L^{2}(M,d_{M}^{1/2})$ as long as $s>\dim M/2$. The latter condition ensures that elements in $H^{s}(M)$ are continuous and hence bounded. 
%This could be useful later on when we make contact with the BvM theorem, since elements in $H^{-1/2(s)}(M)$ can be thought of as of the form $e_{M}d^{-1/2}\varphi+u$, where $\varphi\in H^{s+1/2}(M)$ and $u\in H^{s}_{M}(S)$ or $H^{s-0}_{M}(S)$ (nice boundary behaviour). See \cite[Theorem 5.4]{Grubb2}.
}
\end{Remark}

We are now ready to give the proof of Theorem \ref{thm:mainiso}.

\begin{proof}[Proof of Theorem \ref{thm:mainiso}] By Theorem \ref{theorem:fred}, the map $r_{M}P$ is Fredholm with finite dimensional kernel and co-kernel independent of $s$; in fact elements in the kernel must be in $\mathcal{E}_{-1/2}(M)$, cf. \cite[Theorem 3.5]{Grubb1}.
Hence, it suffices to check that these kernel and co-kernel are trivial. We begin by proving that the kernel is trivial.

Suppose there is $u\in \mathcal{E}_{-1/2}(M)$ such that $r_{M}Pu=0$. Writing $u=e_{M}d_{M}^{-1/2}\varphi$ with $\varphi\in C^{\infty}(M)$, we see that $Nf=0$ where $f=d_{M}^{-1/2}\varphi \in L^{2}(M,d_{M}^{1/2})$  by Lemma \ref{lem:phiL1}. But $I^*_{w}I=d_{M}^{-1/2}N$, hence $I^*_{w}I f=0$.
This implies $(I^*_{w}I f,f)_{L^{2}(M,d_{M}^{1/2})}=0$ and hence $If=0$. To show that $f$ must in fact be smooth, extend $f$ by zero to $U_{1}$ and call the extension $f_{1}$. By Lemma \ref{lem:phiL1}, $f_1\in L^1(U_1)$ so that, using Santal\'o's formula, it is easy to see that $I_1 f_1$ makes sense in $L^1(\partial_+ SU_1)$ and also that $N_1 f_1 = I_1^* I_1 f$ makes sense in $L^1(U_1)$. 
Then $I_{1}f_{1}=0$ and thus $N_{1}f_{1}=0$. Since $N_{1}$ is elliptic,
the function $f_{1}$ must be smooth in $U_{1}$ and hence $f$ is smooth in $M$. Now we use the standard injectivity result for $I$ acting on smooth functions on a simple manifold \cite{Mu}
to conclude that $f=0$.

Let us now check that the co-kernel of $r_{M}P$ is trivial. 
%There may be a more efficient way to do this, but the following seems to work and it is just gymnastics with duals. 
Consider the injection
\[\iota: H^{s}_{M}(S)\hookrightarrow H^{-1/2(s)}(M),\]
where $H_{M}^{s}(S)$ consists of elements in $H^{s}(S)$ with support in $M$ (cf. Subsection \ref{section:S+H}). Let us compute $(r_{M}P\iota)^*$.
The point of using $\iota$ is to end up with standard dualities not involving the H\"ormander spaces. Note
\[(r_{M}P \iota)^*:(H^{s+1}(M))^*\to (H^{s}_{M}(S))^*\]
 where
 \[(H^{s+1}(M))^*=H^{-s-1}_{M}(S),\]
 \[(H^{s}_{M}(S))^*=H^{-s}(M)\]
 are the standard dualities. Take $u\in H^{-s-1}_{M}(S)$ and $f\in H^{s}_{M}(S)$ and observe
 \begin{align}
     (r_{M}Pf,u)_{M}=(Pf,u)_{S}=(f,P_{-s-1}u)_{S}=(f,r_{M}Pu)_{M}.
     \label{eq:tmp}
 \end{align}
 Thus $\iota^* (r_{M}P)^*=(r_{M}P \iota)^*=r_{M}P$. Hence if $u\in H^{-s-1}_{M}(M)$ is such that
 $(r_{M}P)^*u=0$ we see that $r_{M}Pu=0$. By ellipticity $u\in\mathcal {E}_{-1/2}(M)$ and since we have already proved injectivity of $r_{M}P$ on this space we deduce that the co-kernel of $r_{M}P$ is trivial as well. 
 Thus $r_{M}P:\mathcal{E}_{-1/2}(M)\to C^{\infty}(M)$ is a bijection and since $N=r_{M}Pe_{M}$, from the definition of $\mathcal{E}_{-1/2}(M)$ we conclude that $N:d_{M}^{-1/2}C^{\infty}(M)\to C^{\infty}(M)$ is also a bijection.
\end{proof}

\proof[Proof of Theorem \ref{main0}] For the case of $I$, part a) in Theorem \ref{main0} follows immediately from Theorem \ref{thm:mainiso}. Part b) in Theorem \ref{main0} is a direct consequence of Lemma \ref{lemma:weight} and the fact that
$d_{M}^{-1/2}h\in L^{2}(M,d_{M}^{1/2})$ if $h\in C(M)$.

To complete the proof of Theorem \ref{main0} we just need to explain why the same proof works for the attenuated geodesic X-ray transform $I_{a}$.
The microlocal properties of $N_{a}=I_{a}^{*}I_{a}$ are studied in detail in \cite{FSU}. With this in hand, it is straightforward to check that $N_{a}$ will fit the theory developed above. For this we need to extend $a\in C^{\infty}(M)$ smoothly to $S$ and observe that the third proof of Lemma \ref{lem:even} applies to $N_{a}$ since these operators are covered by \cite[Lemma B.1]{DPSU}, see \cite[Section 4]{FSU} for a proof.  Hence the full symbol of $N_{a}$ is also even,
$N_{a,1}$ satisfies the transmission condition with $\mu=-1/2$ with respect to $\partial M$,  and we can derive all the required mapping properties. The proof of Lemma  \ref{lemma:weight} 
works for $I_{a}$ and Theorem \ref{thm:mainiso} holds as well for $I_{a}$ as long as we know that $I_{a}$ is injective on smooth functions.
\qed

\smallskip

{\small{\noindent {\bf Acknowledgments.} The authors are grateful to Gerd Grubb, Plamen Stefanov and Gunther Uhlmann for helpful conversations and comments regarding the microlocal aspects of this paper. The authors also thank Sarah Vall\'elian for sharing insights and code pertaining to \cite{BSV}, and Matteo Giordano and Hanne Kekkonen for proofreading parts of the manuscript. FM was partially funded by NSF grant DMS--1712790. RN was supported by ERC grant UQMSI/647812. GPP was partially funded by EPSRC grant EP/M023842/1.}

\end{document}